\documentclass[10pt,leqno]{article}%[draft]

\usepackage{authblk}
\usepackage{amsmath,amsthm,amssymb,mathtools}
\usepackage{color,tikz-cd,hyperref,geometry}
\hypersetup{colorlinks=true,linkcolor=blue,citecolor=blue,urlcolor=blue,linktocpage}
\usetikzlibrary{patterns}
\newtheorem{thm}{Theorem}[section]
\newtheorem{prp}[thm]{Proposition}
\newtheorem{lem}[thm]{Lemma}
\newtheorem{cor}[thm]{Corollary}

\newtheorem{thm-intro}{Theorem}
\newtheorem{cor-intro}[thm-intro]{Corollary}

\theoremstyle{definition}
\newtheorem{dfn}[thm]{Definition}
\newtheorem{exm}[thm]{Example}
 
\theoremstyle{remark}
\newtheorem{rmk}[thm]{Remark}

\newcommand{\vb}{\,|\,}
\newcommand{\uline}{\rule{0.3cm}{0.1mm}}

\newcommand{\Z}{\mathbb{Z}}
\newcommand{\Q}{\mathbb{Q}}
\newcommand{\R}{\mathbb{R}}

\newcommand{\RP}{\mathbb{RP}}
\newcommand{\cA}{\mathcal{A}}
\newcommand{\cB}{\mathcal{B}}

\newcommand{\cC}{\mathcal{C}}
\newcommand{\cD}{\mathcal{D}}
\newcommand{\cE}{\mathcal{E}}

\newcommand{\cK}{\mathcal{K}}

\newcommand{\cW}{\textup{WFuk}}

\newcommand{\cM}{\mathcal{M}}
\newcommand{\cN}{\mathcal{N}}

\newcommand{\cY}{\mathcal{Y}}

\newcommand{\op}{^\textup{op}}
\newcommand{\colim}{\textup{colim}}
\newcommand{\holim}{\textup{holim}}
\newcommand{\hocolim}{\textup{hocolim}}

\newcommand{\Mod}{\textup{Mod}}

\newcommand{\mSh}{\mu\textup{Sh}}
\newcommand{\Hom}{\textup{Hom}}
\newcommand{\RHom}{\textup{RHom}}

\newcommand{\Perf}{\textup{Perf}}

\newcommand{\Tw}{\textup{Tw}}

\newcommand{\Cyl}{\textup{Cyl}}

\newcommand{\dgCat}{\textup{dgCat}_k}

\newcommand{\Ob}{\textup{Ob}}
\newcommand{\Ho}{\textup{Ho}}
\newcommand{\Top}{\textup{Top}}
\newcommand{\pt}{\textup{pt}}
\newcommand{\obj}{\textup{obj}}
\newcommand{\dgqe}{\dgCat^{\textup{qe}}}
\newcommand{\dgeq}{\dgCat^{\textup{tr}}}
\newcommand{\dgmo}{\dgCat^{\textup{mo}}}

\title{Homotopy Colimits of DG Categories and Fukaya Categories}
\author[$\dagger$]{Dogancan Karabas}
\author[$\star$]{Sangjin Lee}
\affil[$\dagger$]{Department of Mathematics, Northwestern University, 2033 Sheridan Rd, Evanston, IL 60208, United States} 
\affil[$\star$]{Center for Geometry and Physics, Institute for Basic Science (IBS), Pohang 37673, Korea}

\begin{document}

\maketitle

\begin{abstract}	
	We construct a new cylinder object for semifree differential graded (dg) categories in the category of dg categories. 
	Using this, we give a practical formula computing homotopy colimits of semifree dg categories.
	Combining it with the result of Ganatra, Pardon, and Shende, we get a formula computing wrapped Fukaya categories of Weinstein manifolds using their sectorial coverings.
	This formula has lots of applications including a practical computation of the wrapped Fukaya category of any cotangent bundle or plumbing space.
	In this paper, we compute wrapped Fukaya categories of cotangent bundles of lens spaces using their Heegaard decomposition.
	From the computation, we show that the endomorphism algebra of the cotangent fibre is a full invariant of the homotopy type of lens spaces.
\end{abstract}

\tableofcontents

\section{Introduction}

\subsection{Algebraic Introduction}
\label{subsect algebraic introduction}
In this paper, we work with the category of differential graded (dg) categories after inverting quasi-equivalences, pretriangulated equivalences, or Morita equivalences. Many facts are known about these localisations of the category of dg categories. In particular, there is a nice description for homotopy limits when quasi-equivalences are inverted. However, homotopy colimits (even colimits) remained hard to compute in practice. Here, we give a practical formula for the homotopy colimit of semifree dg categories, which are roughly dg categories whose underlying algebras are free.
		
There is the notion of cylinder object, which can be used to describe homotopy colimits and other constructions in localised categories. For an object $C$ in a category $\cC$, a cylinder object for $C$ is defined as another object in $\cC$ satisfying the properties given in Definition \ref{dfn:cylinder}. Because of those properties, cylinder objects are very useful when calculating homotopy colimits.
	
In this paper, we give a construction of a cylinder object for a given semifree dg category. With our construction, we can formulate a homotopy colimit formula when the inputs of the homotopy colimit are semifree dg categories.

Not only is our gluing formula easy to compute but also it presents the homotopy colimit as a semifree dg category. The semifreeness of the resulting category allows us to study it combinatorially assuming it has finitely many generators. A demonstration of this remark can be seen in the second part of this paper, where we study semifree dg categories (associated to lens spaces) via dg functors from them to a simpler dg category. One needs ``zig-zags'' of dg functors instead of just dg functors if the categories are not semifree.

\subsection{Symplectic Introduction}
\label{subsect introduction}
In \cite{gps2}, Ganatra, Pardon, and Shende introduced a way to compute the wrapped Fukaya category of a given Weinstein manifold $W$ by using a Weinstein sectorial covering of $W$. 
One can glue the wrapped Fukaya categories of the Weinstein sectors in the covering, then the resulting category is pretriangulated equivalent to the wrapped Fukaya category of the original Weinstein manifold $W$.

Here, the categories are glued via homotopy colimit. As described in Section \ref{subsect algebraic introduction}, there is a computational difficulty when calculating homotopy colimits. We resolve it by using our cylinder objects and provide a computable homotopy colimit formula.

In Section \ref{sec:wrap-lens}, we give an application of our homotopy colimit formula. We calculate the wrapped Fukaya category of the cotangent bundle of $L(p,q)$ for any relatively prime $p>q\geq 1$ using its Heegaard decomposition, where $L(p,q)$ is a lens space, which is a quotient of $S^3$ by a $\Z_p$-action.

We describe the endomorphism algebra of a cotangent fibre of $L(p,q)$ in the wrapped Fukaya category. This description provides some results, including that the quasi-equivalence class of the endomorphism algebra of a cotangent fibre of $L(p,q)$ is a full invariant of the homotopy type of $L(p,q)$, and that the wrapped Fukaya category of $T^*L(p,q)$ is an invariant of the homotopy type of $L(p,q)$ up to pretriangulated equivalence.
We also get that the $A_\infty$-structure of the based loop space of $L(p,q)$ distinguishes non-homotopic lens spaces.

More detailed statements of our results will appear in Section \ref{subsect structure and results}.

\subsection{Results and the Structure of the Paper}
\label{subsect structure and results}
Here, we list the main results of this paper. For more precise statements and the proofs, see the corresponding sections.

\begin{thm}[Definition \ref{dfn:cylinder-dg}, Theorem \ref{thm:cylinder}]\label{thm:cylinder-intro}
	Let $\cC$ be a semifree dg category (see Definition \ref{dfn:semifree-dg-category}) with generating morphisms, and let $\cC_1, \cC_2$ be two copies of $\cC$. We define a cylinder object $\Cyl(\cC)$ for $\cC$, which is a semifree dg category obtained by adding morphisms freely to $\cC_1\amalg\cC_2$ in algebra level. The added morphisms consist of
	\begin{itemize}
		\item closed degree zero morphisms $t_A\colon A_1\to A_2$ for each $A\in\cC$ where $A_1$ and $A_2$ are corresponding objects in $\cC_1$ and $\cC_2$, respectively,
		\item degree $|f|-1$ morphisms $t_f\colon A_1\to B_2$ for each generating morphism $f\in\hom^*_{\cC}(A,B)$,
	\end{itemize}
	and then we invert the morphism $t_A$ for each $A\in\cC$. The differential of the added morphisms $t_f$ are given in Definition \ref{dfn:cylinder-pre-dg}.
\end{thm}

\begin{proof}[Idea of proof]
	A cylinder object for the dg category $\cC$ can be obtained by adding morphisms to $\cC_1\amalg\cC_2$ freely in algebra level, and they can have nontrivial differentials (this is called a semifree extension). Also, this new dg category must be quasi-equivalent to $\cC$. For that reason, we add an invertible morphism $t_A$ for each $A\in\cC$ to identify the objects of $\cC_1$ and $\cC_2$. To identify the morphisms in $\cC_1$ and $\cC_2$, we impose the condition
	\[f_2 t_A = t_B f_1,\]
	where $f\in\hom^*_{\cC}(A,B)$, and $f_1$ and $f_2$ are corresponding morphisms in $\cC_1$ and $\cC_2$, respectively. However, this contradicts with freeness. Hence, we add the morphisms $t_f$ such that
	\[dt_f=(-1)^{|f|}(f_2 t_A - t_B f_1) .\]
	Then we have to impose more commutativeness conditions involving the morphisms $t_f$ to identify the morphisms in $\cC_1$ and $\cC_2$. Hence, we would add more morphisms (possibly infinitely many) to preserve freeness (compare with the definition of $A_{\infty}$-natural transformations). But instead of adding more new morphisms, we perturb the differential of $t_f$'s as in Definition \ref{dfn:cylinder-pre-dg} so that there is no need for more morphisms. This is thanks to the semifreeness of $\cC$ whose morphisms are freely generated in algebra level by some collection of morphisms. Hence, by adding $t_A$'s and $t_f$'s, the morphisms of $\cC_1$ and $\cC_2$ are identified, and $\Cyl(\cC)$ is indeed quasi-equivalent to $\cC$.
\end{proof}

Two immediate uses of the cylinder object are as follows:

Note that there are the inclusion functors
\begin{align*}
	i_1&\colon \cC=\cC_1\hookrightarrow \cC_1\sqcup \cC_2 \hookrightarrow \Cyl(\cC),\\
	i_2&\colon \cC=\cC_2\hookrightarrow \cC_1\sqcup \cC_2 \hookrightarrow \Cyl(\cC) .
\end{align*}

Assume that we invert quasi-equivalences in the category of dg categories $\dgCat$ to get a new category $\dgqe$. Consider two dg functors
\[F_1, F_2\colon \cC\to\cD,\]
where $\cC$ is a semifree dg category. If there exists a dg functor
\[H\colon\Cyl(\cC)\to\cD,\]
making the following diagram commute
\[\begin{tikzcd}
	\cC\ar[r,"i_1"]\ar[rd,"F_1"'] & \Cyl(\cC)\ar[d,"H"] & \cC\ar[l,"i_2"']\ar[ld,"F_2"]\\
	& \cD
\end{tikzcd}\]
then $F_1$ and $F_2$ are naturally isomorphic in $\dgqe$. If we invert pretriangulated equivalences (resp.\ Morita equivalences) instead, we need to replace $\cD$ with its pretriangulated closure (resp.\ idempotent completion of its pretriangulated closure) of $\cD$ for this statement to be true.

Since $\Cyl(\cC)$ is a semifree dg category, to construct $H$, we only need to state the image of the objects and generating morphisms of $\cC$. Since our $\Cyl(\cC)$ construction does not involve too many extra morphisms, it is relatively easy to analyse and construct the possible $H$'s. 

The second use is our homotopy colimit formula:

\begin{thm}[Theorem \ref{thm:hocolim}]\label{thm:hocolim-intro}
	Let $\cA, \cB, \cC$ be semifree dg categories, $\alpha\colon\cC\to\cA$ and $\beta\colon\cC\to\cB$ be dg functors. The homotopy colimit
	\begin{equation}\label{eq:hocolim-intro}
		\hocolim\left(
		\begin{tikzcd}
			\cA & & \cB\\
			& \cC\ar[lu,"\alpha"]\ar[ru,"\beta"']
		\end{tikzcd}
		\right)
	\end{equation}
	can be given by a semifree dg category which is obtained by adding morphisms freely to $\cA\amalg\cB$ consisting of
	\begin{itemize}
		\item closed degree zero morphisms $t_C\colon \alpha(C)\to \beta(C)$ for each $C\in\cC$,
		\item degree $|f|-1$ morphisms $t_f\colon \alpha(A)\to \beta(B)$ for each generating morphism $f\in\hom^*_{\cC}(A,B)$
	\end{itemize}
	then inverting the morphisms in $\{t_C\vb C\in\cC\}$. The differentials of the added morphisms $t_f$ are given in Theorem \ref{thm:hocolim}.
	
	Moreover, taking homotopy colimit commutes with the localisation in the following sense (see Theorem \ref{thm:hocolim-loc}): If a collection of morphisms $S_{\cA}$, $S_{\cB}$, and $S_{\cC}$ are inverted in $\cA$, $\cB$, and $\cC$, respectively, then the new homotopy colimit is given by the homotopy colimit (\ref{eq:hocolim-intro}), and then by inverting the morphisms in $S_{\cA}$ and $S_{\cB}$.
\end{thm}

\begin{proof}[Idea of proof]
	One can convert a homotopy colimit diagram to a usual colimit diagram via
	\[\hocolim\left(
	\begin{tikzcd}
		\cA & & \cB\\
		& \cC\ar[lu,"\alpha"]\ar[ru,"\beta"']
	\end{tikzcd}
	\right)
	\simeq
	\colim\left(
	\begin{tikzcd}[column sep=0.3cm]
		\cA & & \Cyl(\cC) & & \cB\\
		& \cC\ar[lu,"\alpha"]\ar[ru,"i_1"'] & & \cC\ar[lu,"i_2"]\ar[ru,"\beta"'] 
	\end{tikzcd}
	\right) .
	\]
	The calculation of colimits is described by Proposition \ref{prp:colimit}. Using the cylinder object $\Cyl(\cC)$ given in Theorem \ref{thm:cylinder-intro}, we get the result by direct computation.
\end{proof}

This is an algebraic tool, but it also applies to symplectic geometry thanks to the following result.

\begin{thm}[\cite{gps2}]\label{thm:gps-intro}
	Let $W=W_1\cup W_2$ be a Liouville manifold such that $W_1$ and $W_2$ are Weinstein sectors meeting along a hypersurface in $W$. If the neighboorhood of the hypersurface is $F\times T^*[0,1]$ where $F$ is a Weinstein sector up to a deformation, then we have
	\[\cW(W)\simeq\hocolim\left(\begin{tikzcd}[column sep=0.1cm]
		\cW(W_1) & & \cW(W_2)\\
		& \cW(F)\ar[lu]\ar[ru]
	\end{tikzcd}\right),\]
	up to pretriangulated equivalence, where $\cW(X)$ is the wrapped Fukaya category of $X$.
\end{thm}

Theorem \ref{thm:gps-intro} also induces a gluing property for the chains on the based loop spaces using Abouzaid's work (Theorem \ref{thm:wrapped-loops}).

\begin{thm}[Theorem \ref{thm:loop-hocolim}]\label{thm:loop-hocolim-intro}
	Let $M=M_1\cup M_2$ be a connected smooth manifold such that $M_1$, $M_2$, and $M_1\cap M_2$ are connected smooth manifolds and open in $M$. Let $x\in M_1\cap M_2$ be a point. Then we have
	\[C_{-*}(\Omega_x M)\simeq\hocolim\left(\begin{tikzcd}[column sep=0.1cm]		
		C_{-*}(\Omega_x M_1) & & C_{-*}(\Omega_x M_2)\\
		& C_{-*}(\Omega_x (M_1\cap M_2))\ar[lu]\ar[ru]
	\end{tikzcd}\right),\]
	up to quasi-equivalence, where $C_{-*}(\Omega_xM)$ is the chains on the based loop space of $M$ (at $x\in M$), which is a dg algebra (dga).
\end{thm}

One important remark is that there is a procedure called ``arborealisation'' (see \cite{arboreal}, \cite{arboreal-starkston}, \cite{arboreal-alvarez}), which allows us to describe the wrapped Fukaya categories of the pieces we glue by some semifree dg categories. Assuming that the gluing maps can also be described by morphisms of semifree dg categories, the problem of calculating wrapped Fukaya category reduces to taking homotopy colimit of semifree dg categories, which is exactly what our Theorem \ref{thm:hocolim-intro} handles. This means that, our homotopy colimit formula directly applies to, at least, cotangent bundles and plumbing spaces.  In this paper, we focused on the cotangent bundles of lens spaces, and we got the following result.

\begin{thm}[Theorem \ref{thm:wrap-lens}, Theorem \ref{thm:loop-lens}]\label{thm:lens-dga-intro}
	If $L(p,q)$ is a lens space with $p>q\geq 1$ and $(p,q)=1$, then
	\[\cW(T^*L(p,q))\]
	is generated by a cotangent fibre, whose endomorphism algebra $\cC_{p,q}$ is a semifree dg algebra generated by the three morphisms $x,y,z$ with the degrees
	\[|x|=0,\qquad |y|=-1,\qquad |z|=-2 ,\]
	and with the differentials
	\begin{align*}
		dx&=0,\\
		dy&=1-x^p,\\
		dz&=x^q y-yx^q .
	\end{align*}

	Also, the chains $C_{-*}(\Omega_s L(p,q))$ on the based loop space of the lens space $L(p,q)$ is quasi-equivalent to $\cC_{p,q}$.
\end{thm}

\begin{proof}[Idea of proof]
	Lens spaces can be presented by a gluing of two solid tori along their boundaries. Then, Theorem \ref{thm:gps-intro} allows us to calculate it using the wrapped Fukaya category of solid tori and a torus. Consequently, our homotopy colimit formula (Theorem \ref{thm:hocolim-intro}) gives the result after some simplification.
\end{proof}

\begin{thm}[Theorem \ref{thm homotopy type to quasi}, Theorem \ref{thm quasi to homotopy type}]
	\label{thm:lens space}
	If we work with $\Z$ coefficient, the endomorphism algebras $\cC_{p,q_1}$ and $\cC_{p,q_2}$ (of $L(p,q_1)$ and $L(p,q_2)$, respectively) are quasi-equivalent if and only if $L(p,q_1)$ and $L(p,q_2)$ are homotopy equivalent.
	
	In particular, the endomorphism algebra completely distinguishes homotopy type of lens spaces, but it doesn't detect their simple homotopy type.
	
	This also means that the loop spaces of non-homotopic lens spaces are not isomorphic as $A_{\infty}$-spaces.
	
	If we work with field coefficients with characteristic not equal to $p$, the endomorphism algebras $\cC_{p,q_1}$ and $\cC_{p,q_2}$ are quasi-equivalent for any $L(p,q_1)$ and $L(p,q_2)$.
\end{thm}

\begin{rmk}
	The cohomology of $\cC_{p,q}$ (even with the product structure) does not distinguish non-homotopic lens spaces for fixed $p$. See Remark \ref{rmk:cohomology-lens} and Remark \ref{rmk:product-lens}.
\end{rmk}

\begin{proof}[Idea of proof]
	Since we presented the endomorphism algebra $\cC_{p,q}$ as a semifree dg algebra, any functor from $\cC_{p,q}$ can be regarded as a dg functor (so, there is no need to deal with quasi-functors/$A_{\infty}$-functors). We look at the dg functors from $\cC_{p,q}$ to some simple dga (see $\cD$ in Section \ref{sec:notation-convention}) and by analysing the possible dg functors between them, we distinguish non-homotopic lens spaces.
	
	As for homotopic lens spaces, we can explicitly create a dg functor between their endomorphism algebras. Since this dg functor is quite complicated, we construct a quasi-faithful dg functor from $\cC_{p,q}$ to another simple dga (see $\cE$ in Section \ref{sec:notation-convention}) to analyse this dg functor and show that it is a quasi-equivalence.
\end{proof}

\begin{cor}[Corollary \ref{cor Fukaya category}]
	The wrapped Fukaya category of $T^*L(p,q)$ is an invariant of the homotopy type of $L(p,q)$.
\end{cor}

Remark \ref{rmk:wrapped-lens} explains the meaning of Theorem \ref{thm:lens space} in symplectic topology and homological mirror symmetry.

The current paper consists of two parts.
Sections \ref{sec:cylinder-object and homotopy-colimit} and \ref{sec:wrap-lens} are in the first part.
We give our construction of cylinder objects and homotopy colimits in Section \ref{sec:cylinder-object and homotopy-colimit}. The preliminaries for model categories, dg categories, semifree extensions are also given in Section \ref{sec:cylinder-object and homotopy-colimit}. The reader can find proofs of Theorems \ref{thm:cylinder-intro} and \ref{thm:hocolim-intro} there. The preliminaries for wrapped Fukaya categories and chains on based loop spaces are given in Section \ref{sec:wrap-lens}. Then, Section \ref{sec:wrap-lens} uses the homotopy colimit formula (Theorem \ref{thm:hocolim-intro}) on the wrapped Fukaya category of cotangent bundles of some 3-manifolds, and on the chains on based loop spaces. In particular, we get Theorems \ref{thm:loop-hocolim-intro} and \ref{thm:lens-dga-intro}.

The second part studies the differential graded algebras given in Theorem \ref{thm:lens-dga-intro}.
As a result, we prove Theorem \ref{thm:lens space}.

\subsection{Notations and Conventions}
\label{sec:notation-convention}

The following notations will be introduced in the paper as they are required, but here we will list the important ones for the sake of easy reference. Note that we will read compositions of morphisms right-to-left.

\begin{itemize}
	\item dg: differential graded,
	\item dga: differential graded algebra,
	\item $\Z_p$: multiplicative group of integers modulo $p$,
	\item $\Ob\,\cC$: objects of a category $\cC$,
	\item $\Hom_{\cC}(A,B)$ (or just $\Hom(A,B)$): morphisms from the object $A$ to $B$ in a category $\cC$,
	\item $\hom^*_{\cC}(A,B)$: the (co)chain complex of the morphisms from the object $A$ to $B$ in a dg category $\cC$,
	\item $\Hom^*_{\cC}(A,B)$: the cohomology of $\hom^*_{\cC}(A,B)$,
	\item $|f|$: degree of a morphism $f$,
	\item $df$: differential of a morphism $f$,
	\item $1_C$ (or just $1$): the identity morphism from $C$ to $C$,
	\item $H^*\cC$: the graded homotopy category of a dg category $\cC$,
	\item $H^*F\colon H^*\cC\to H^*\cD$: the induced functor between graded homotopy categories for a given dg functor $F$,
	\item $\dgCat$: the category of (small, $k$-linear) dg categories,
	\item $\dgqe$: a model structure for $\dgCat$, where the weak equivalences are quasi-equivalences,
	\item $\dgeq$: a model structure for $\dgCat$, where the weak equivalences are pretriangulated equivalences,
	\item $\dgmo$: a model structure for $\dgCat$, where the weak equivalences are Morita equivalences,
	\item $\cC[S^{-1}]$: the (dg) localisation of a (dg) category $\cC$ at a collection of morphisms $S$,
	\item $\Ho(\cC)$: the homotopy category of a category $\cC$ with the weak equivalences, i.e. $\cC[W^{-1}]$ where $W$ is the set of weak equivalences,,
	\item $[A,B]$: morphisms from the object $A$ to $B$ in the homotopy category of a model category,
	\item $\Mod\,k$: the dg category of (co)chain complexes of $k$-modules, localised at quasi-isomorphisms,
	\item $\RHom(\cC,\cD)$: the internal Hom between the dg categories $\cC$ and $\cD$,
	\item $\Mod\,\cC$: the internal Hom $\RHom(\cC^{\op},\Mod\,k)$,
	\item $\Tw\,\cC$: the dg category of twisted complexes in a dg category $\cC$,
	\item $\Perf\,\cC$: the split-closure of $\Tw\,\cC$,
	\item $\Cyl(\cC)$: a cylinder object for $\cC$,
	\item $\hocolim$: homotopy colimit,
	\item $\holim$: homotopy limit,
	\item $T^*M$: the cotangent bundle of a smooth manifold $M$,
	\item $\cW(W)$: the wrapped Fukaya category of a Liouville sector $W$,
	\item $\mSh(W)$: the dg category of (unbounded) microlocal sheaves on the skeleton of a Weinstein manifold $W$,
	\item $C_{-*}(\Omega_xM)$: chains on the based loop space of $M$ at $x$,
	\item $S^n$: $n$-dimensional sphere,
	\item $\Sigma_g$: genus $g$ surface,
	\item $L(p,q)$: a lens space with $p>q\geq 1$ and $(p,q)=1$,
	\item $k\langle x_1,\ldots,x_n\rangle$: a semifree dga with the generating morphisms $x_1,\ldots,x_n$,
	\item $k\langle x_1,\ldots,x_n\rangle[\{y_1,\ldots,y_m\}^{-1}]$: a semifree dga $k\langle x_1,\ldots,x_n\rangle$ where the morphisms $y_1,\ldots,y_m$ are inverted,
	\item $\cC_{p,q}$: the endomorphism algebra of a cotangent fibre of the lens space $L(p,q)$ in the wrapped Fukaya category,
	\item $H_{p,q}$: the cohomology of $\cC_{p,q}$,
	\item $\chi, \Lambda, f_n(x)$: elements of $\cC_{p,q}$ defined in Section \ref{subsect notation},
	\item $\mathcal{E}$: the dga $k\langle \alpha,\gamma\rangle/\{\alpha^p=1,\alpha\gamma=\gamma\alpha\}$ with $|\alpha|=0$, $|\gamma|=-2$, and zero differential,
	\item $\pi\colon\cC_{p,q}\to\mathcal{E}$: the dga morphism defined by Equation (\ref{eqn pi}),
	\item $B_N$: a basis for $\cC_{p,q}^{-N}$ given by Equation (\ref{eqn basis}),
	\item $\mathfrak{A}_i,\mathfrak{B}_{m,i},\mathfrak{C}_i,\mathfrak{D}_{m,i}$: the conditions defined in Definition \ref{dfn conditions},
	\item $\Psi,\Phi$: $\Z$-module morphisms defined in Definition \ref{def Psi and Phi},
	\item $\cD$: the semifree dga $k\langle\beta,\gamma\rangle$ with $|\beta|=-1$, $|\gamma|=-2$, and zero differential,
	\item $\mu_{a,b,c}\colon\cC_{p,q}\to\cD$: the dga morphism defined by Equation (\ref{eqn dga morphism}),
	\item $\mu$: the dga morphism $\mu_{1,1,0}$.
\end{itemize}

\subsection{Acknowledgements}
\label{subsect acknoledgment} 
We thank Emmy Murphy, Eric Zaslow, and Dongwook Choa for helpful discussions and valuable comments. We also thank Shanon Rubin for his careful reading of the paper and pointing out mistakes.

The first (resp.\ second) named author had (resp.\ has) been supported by the Institute for Basic Science (IBS-R003-D1).

\clearpage

\part{Cylinder Objects and Homotopy Colimits}

\section{Cylinder Objects and Homotopy Colimits in the Category of Dg Categories}
\label{sec:cylinder-object and homotopy-colimit}
Let $k$ be a commutative ring (with a unit) throughout the chapter. Our goal is to describe a simple formula for cylinder objects and homotopy colimits for semifree dg categories. Also, we will show that our construction commutes with the localisation of dg categories.

\subsection{Model Categories}

We will recall some facts in the theory of model categories. Our main references are \cite{hovey} and \cite{hirschhorn-model}. The main use of model categories for us is understanding categorical constructions in the localisations of categories at ``weak equivalences'' via the initial category and the auxillary data, cofibrations and fibrations, which can be sometimes roughly thought as ``nice inclusions'' and ``nice surjections'', respectively. We will be concerned with the category of dg categories localised at various collections of morphisms, hence this understanding is crucial.

\begin{dfn}\mbox{}
	\begin{enumerate}
		\item A \textit{(closed) model category} $\cM$ is a complete and cocomplete category with three classes of morphisms, called \textit{weak equivalences}, \textit{cofibrations} and \textit{fibrations}, satisfying the axioms given in \cite{hirschhorn-model}.
		\item We call the data of weak equivalences, cofibrations, and fibrations a \textit{model structure} on $\cM$.
		\item A cofibration (resp.\ fibration) which is also a weak equivalence is called an \textit{acyclic cofibration} (resp.\ \textit{acyclic fibration}).
	\end{enumerate}
\end{dfn}

\begin{rmk}
	The axioms in \cite{hirschhorn-model} imply that any two of the three classes of the morphisms (weak equivalences, fibrations, cofibrations) determine the third.
\end{rmk}

\begin{rmk}
	If $\cM$ is a model category, $\cM^{\op}$ is also a model category with the same weak equivalences. Cofibrations of $\cM^{\op}$ are fibrations of $\cM$ and vice versa.
\end{rmk}

\begin{dfn}
	Let $C$ be an object in a model category $\cM$. $C$ is called \textit{cofibrant} (resp.\ \textit{fibrant}) if the morphism from the initial (resp.\ final) object of $\cM$ to $C$ is a cofibration (resp.\ fibration).
\end{dfn}

\begin{rmk}
	By the model category axioms, for any object $C$ in $\cM$, there is a weak equivalence from a cofibrant object to $C$, and a weak equivalence from $C$ to a fibrant object.
\end{rmk}

\begin{dfn}
	A model category $\cM$ is \textit{cofibrantly generated} if there is a set of cofibrations, called \textit{generating cofibrations}, and a set of acyclic cofibrations, called \textit{generating acyclic cofibrations}, which satisfy the conditions given in \cite{hovey} and determine the model structure of $\cM$. In particular, any cofibration in $\cM$ is given by a retract of a transfinite composition of cobase changes of coproducts of generating cofibrations.
\end{dfn}

\begin{exm}\label{exm:top-model}
	Let $\cM=\Top$ be the category of topological spaces. It has a cofibrantly generated model structure where the weak equivalences are the weak homotopy equivalences, and generating cofibrations are the boundary inclusions $S^{n-1}\to D^n$ for all $n\geq 0$. Here, all objects are fibrant, and in particular, CW complexes are cofibrant.
\end{exm}

\begin{dfn}
	A \textit{localisation} of a category $\cM$ at a class of morphisms $W$ in $\cM$ is obtained from $\cM$ by inverting the morphisms in $W$. More precisely, it is the category $\cM[W^{-1}]$ with the functor $l\colon \cM\to \cM[W^{-1}]$ such that for any category $\cN$, 
	\begin{itemize}
		\item the induced morphism
		\[l^*\colon \Hom(\cM[W^{-1}],\cN)\to\Hom(\cM,\cN)\]
		is injective, and
		\item the image of $l^*$ consists of all functors sending each morphism in $W$ to an isomorphism $\cN$.
	\end{itemize} 
\end{dfn}

\begin{rmk}
	If a localisation exists, it is unique up to a unique isomorphism.
\end{rmk}

\begin{prp}
	If $\cM$ is a model category with the weak equivalences $W$, then the localisation $\cM[W^{-1}]$ exists, and it is unique up to a unique isomorphism.
\end{prp}

There is a way of obtaining a new model category from an initial model category by extending the class of weak equivalences.

\begin{dfn}\label{dfn:bousfield}
	Let $\cM$ be a model category with the weak equivalences $W$, and let $\overline W\supset W$ be a class of morphisms in $\cM$ given by $S$-local equivalences for some class of morphisms $S$ in $\cM$ (see \cite{hirschhorn-model} for the definition and the details). A \textit{left Bousfield localisation} of $\cM$ is a new model structure on $\cM$ with the same cofibrations and with the weak equivalences $\overline W$.
\end{dfn}

Now we are ready to define the category of our interest.

\begin{dfn}\label{dfn:homotopy-category}
	Let $\cM$ be a model category with the weak equivalences $W$. 
	\begin{enumerate}
		\item The \textit{homotopy category} of the model category $\cM$, denoted by $\Ho(\cM)$, is given by the localisation $\cM[W^{-1}]$.
		\item We write
		\[[A,B]:=\Hom_{\Ho(\cM)}(A,B) .\]
		\item We call the objects $A$ and $B$ in $\cM$ \textit{weakly equivalent} (or \textit{the same up to weak equivalence}) if they are isomorphic in $\Ho(\cM)$.
	\end{enumerate}
\end{dfn}

\begin{prp}\label{prp:weak-equiv}
	Let $f\colon A\to B$ be a morphism in a model category $\cM$. $f$ is isomorphism in $\Ho(\cM)$ if and only if $f$ is a weak equivalence in $\cM$.
\end{prp}

To understand the homotopy category, the following will be our main tool.

\begin{dfn}\label{dfn:cylinder}
	Let $\cM$ be a model category and $C,C'\in\cM$. $C'$ is a \textit{cylinder object} for $C$ if we have a decomposition of the codiagonal map
	\[\nabla\colon C\amalg C\xrightarrow{i}C'\xrightarrow{p}C,\]
	where $i$ is a cofibration, $p$ is a weak equivalence.
\end{dfn}

\begin{rmk}\label{rmk:very-good-cyl}
	Sometimes, this cylinder object is called a \textit{``good'' cylinder object}. If, in addition, $p$ is a fibration, then $C'$ is called a \textit{``very good'' cylinder object} for $C$. The existence of a ``very good'' cylinder object for any given object is guaranteed by model category axioms.
\end{rmk}

\begin{dfn}
	Let $\cM$ be a model category, and $f_1,f_2\in\Hom_{\cM}(A,B)$.
	\begin{enumerate}
		\item We say $f_1$ and $f_2$ are \textit{left homotopic} if there exists a morphism $h\colon A'\to B$ for a cylinder object $A'$ for $A$ such that the diagram
		\[\begin{tikzcd}
			A\ar[r,"i_1"]\ar[rd,"f_1"'] & A'\ar[d,"h"] & A\ar[l,"i_2"']\ar[ld,"f_2"]\\
			& B
		\end{tikzcd}\]
		commutes, where $i_1\amalg i_2\colon A\amalg A\to A'$ is the cofibration for the cylinder object $A'$.
		\item The morphism $h$ above is called a \textit{left homotopy} from $f_1$ to $f_2$,
	\end{enumerate} 
\end{dfn}

In general, localisations are hard to describe. However, for a model category, we have a nice description of its localisation at weak equivalences, i.e.\ its homotopy category, in terms of the initial category.

\begin{prp}
	Let $\cM$ be a model category, and let $\cM_{cf}$ be its full subcategory of objects which are both cofibrant and fibrant in $\cM$. Then, left homotopy gives an equivalence relation on morphisms in $\cM_{cf}$. Moreover, the homotopy category $\Ho(\cM)$ of $\cM$ is equivalent to the quotient of $\cM_{cf}$ by the left homotopy.
\end{prp}

\begin{prp}\label{prp:model-hom}
	Let $\cM$ be a model category. If $A$ is cofibrant and $B$ is fibrant in $\cM$, then there is a natural isomorphism
	\[[A,B]\simeq \Hom_{\cM}(A,B)/\!\sim,\]
	where the equivalence relation $\sim$ is given by $f\sim g$ if $f$ and $g$ are left homotopic.
\end{prp}

The next goal is to describe a type of ``colimit'' (resp.\ ``limit'') for homotopy categories.

\begin{dfn}
	Let $\cM$ be a model category and $\cD$ be a diagram. A \textit{homotopy colimit} (resp.\ \textit{homotopy limit}) of a functor $F\colon\cD\to\cM$ is the image of a cofibrant (resp.\ fibrant) replacement of $F$ under the colimit (resp.\ limit) functor $\Hom(\cD,\cM)\to\cM$. This is well-defined in $\Ho(\cM)$.
\end{dfn}

Note that the category $\Hom(\cD,\cM)$ can be given a model structure, called a Reedy model structure. See e.g.\ \cite{dugger} for the details. After studying cofibrant objects in $\Hom(\cD,\cM)$, we have the following propositions describing homotopy colimits in terms of usual colimits.

\begin{prp}\label{prp:model-hocolim-1}
	If $A,B,C$ are cofibrant objects in a model category $\cM$, then we have the weak equivalence
	\[\hocolim\left(\begin{tikzcd}
		A & & B\\
		& C\ar[lu,"f"]\ar[ru,"g"']
	\end{tikzcd}\right)
	\simeq
	\colim\left(\begin{tikzcd}
		A & & B\\
		& C\ar[lu,"f"]\ar[ru,"g"']
	\end{tikzcd}\right)\]
	in $\cM$, if $f$ is a cofibration.
\end{prp}

\begin{prp}\label{prp:model-hocolim-2}
	If $A,B,C,C_1,C_2$ are cofibrant objects in a model category $\cM$, then we have the weak equivalence
	\[\hocolim\left(\begin{tikzcd}[column sep=0.3cm]
		A & & C & &B\\
		& C_1\ar[lu,"f"]\ar[ru,"h_1"'] & & C_2\ar[lu,"h_2"]\ar[ru,"g"']
	\end{tikzcd}\right)
	\simeq
	\colim\left(\begin{tikzcd}[column sep=0.3cm]
		A & & C & &B\\
		& C_1\ar[lu,"f"]\ar[ru,"h_1"'] & & C_2\ar[lu,"h_2"]\ar[ru,"g"']
	\end{tikzcd}\right)\]
	in $\cM$, if $C_1\amalg C_2\xrightarrow{h_1\amalg h_2}C$ is a cofibration.
\end{prp}

As a direct corollary of Proposition \ref{prp:model-hocolim-2}, we have the following proposition which will be our main tool when calculation homotopy colimits.

\begin{prp}\label{prp:hocolim-to-colim}
	If $A,B,C$ are cofibrant objects in a model category $\cM$, then we have the weak equivalence
	\[\hocolim\left(
	\begin{tikzcd}
		A & & B\\
		& C\ar[lu,"f"]\ar[ru,"g"']
	\end{tikzcd}
	\right)
	\simeq
	\colim\left(
	\begin{tikzcd}[column sep=0.3cm]
		A & & C' & & B\\
		& C\ar[lu,"f"]\ar[ru,"i_1"'] & & C\ar[lu,"i_2"]\ar[ru,"g"'] 
	\end{tikzcd}
	\right)
	\]
	in $\cM$, where $C'$ is a cylindrical object for $C$, and where $i_1\amalg i_2\colon C\amalg C\to C'$ is the cofibration for the cylinder object $C'$.
\end{prp}

\begin{exm}
	Consider the colimit
	\[\colim\left(\begin{tikzcd}
		D^2 & & \pt\\
		& S^1\ar[lu,"\parbox{1.7cm}{\scriptsize inclusion of the boundary}"]\ar[ru]
	\end{tikzcd}\right)
	\simeq S^2\]
	in the category of topological spaces $\cM=\Top$, which is obtained from the disjoint union of a disk $D^2$ and a point $\pt$ by identifying the images of the circle $S^1$, hence the colimit is a sphere $S^2$. If we consider this diagram in $\Ho(\Top)$ instead (for the model structure explained in Example \ref{exm:top-model}), the correct notion of colimit will be homotopy colimit. We will have the weak (homotopy) equivalence
	\[\hocolim\left(\begin{tikzcd}
		D^2 & & \pt\\
		& S^1\ar[lu]\ar[ru]
	\end{tikzcd}\right)
	\simeq 
	\colim\left(\begin{tikzcd}
		D^2 & & \pt\\
		& S^1\ar[lu]\ar[ru]
	\end{tikzcd}\right)
	\simeq S^2\]
	by Proposition \ref{prp:model-hocolim-1}, since the inclusion of the boundary is a cofibration, and since all the objects are CW complexes. We note that every CW complex is cofibrant in $\Top$.
	
	If we replace the disk with a point in the diagram, we should get the same homotopy colimit, because the disk and the point are weakly equivalent. However,
	\[\hocolim\left(\begin{tikzcd}
		\pt & & \pt\\
		& S^1\ar[lu]\ar[ru]
	\end{tikzcd}\right)
	\not\simeq 
	\colim\left(\begin{tikzcd}
		\pt & & \pt\\
		& S^1\ar[lu]\ar[ru]
	\end{tikzcd}\right)
	\simeq \pt,\]
	because the disk and the point are not isomorphic (i.e.\ not homeomorphic) in $\Top$. A way to deal with this issue is to use cylinder objects. In the model category of $\Top$, $C\times [1,2]$ is a cylinder object for $C$ for any object $C$. Then by Proposition \ref{prp:hocolim-to-colim}, we have the weak equivalence
	\[\hocolim\left(
	\begin{tikzcd}[column sep=0.5cm]
		\pt & & \pt\\
			& S^1\ar[lu]\ar[ru]
		\end{tikzcd}
		\right)
		\simeq
		\colim\left(
		\begin{tikzcd}[column sep=0.2cm]
			\pt & & S^1\times [1,2] & & \pt\\
			& S^1\ar[lu]\ar[ru,"i_1"'] & & S^1\ar[lu,"i_2"]\ar[ru] 
		\end{tikzcd}
		\right)
		\simeq S^2\]
		as expected, where $i_n$ is the inclusion of $S^1$ to $S^1\times\{n\}$.
\end{exm}

\subsection{The Category of DG Categories $\dgCat$}

For a survey of dg categories, see \cite{dgcat}, \cite{toen}. We will recall some basics about dg categories in this section. Recall that $k$ is a commutative ring.

\begin{dfn}\mbox{}
	\begin{enumerate}
		\item A \textit{differential graded (dg) category} is a category enriched over the symmetric monoidal category of complexes of $k$-modules. Explicitly, it is a category $\cC$ satisfying the followings;
		\begin{itemize}
			\item For any $A,B\in\cC$, the set of morphisms form $A$ to $B$
			\[\hom^*(A,B)=\bigoplus_{n\in\Z}\hom^n(A,B)\]
			forms a $\Z$-graded (co)chain complex of $k$-modules.
			\item There is a differential map 
			\[d\colon \hom^n(A,B)\to\hom^{n+1}(A,B),\]
			such that $d^2=0$.
			\item For any $A,B,C\in\cC$, the composition map
			\[\hom^*(B,C)\otimes\hom^*(A,B)\to\hom^*(A,C)\]
			is a graded $k$-linear map.
			\item The differential $d$ and the composition satisfy the \textit{(graded) Leibniz rule}
			\[d(gf)=(dg)f+(-1)^{|g|}g(df),\]
			for any homogeneous (i.e.\ of a particular degree) and composable morphisms $f$ and $g$, where $|g|$ denotes the degree of $g$ (i.e.\ $g\in\hom^{|g|}(B,C)$ for some $A,B\in\cC$).
			\item  Finally, for any object $C\in\cC$, the identity morphism $1_C$ is of degree zero, and $d(1_C)=0$.
		\end{itemize}
		\item A dg category with one object is called a \textit{dg algebra (dga)}.
	\end{enumerate}
\end{dfn}

\begin{rmk}
	We are working with $\Z$-graded morphism complexes, but the statements of this paper also hold for dg categories with $\Z/2$-graded morphism complexes by ignoring signs.
\end{rmk}

\begin{dfn}
	Let $\cC$ be a dg category. We write $\Hom^*(A,B)$ for the cohomology of the chain complex $\hom^*(A,B)$ for any $A,B\in\cC$. 
	\begin{enumerate}
		\item The \textit{graded homotopy category} $H^*\cC$ of $\cC$ is the graded $k$-linear category with the same objects as $\cC$ and with the morphism space $\Hom^*(A,B)$ for any $A,B\in\cC$.
		\item The \textit{homotopy category} $H^0\cC$ of $\cC$ is the $k$-linear category with the same objects as $\cC$ and with the morphism space $\Hom^0(A,B)$ for any $A,B\in\cC$.
	\end{enumerate} 
\end{dfn}

\begin{dfn}\mbox{}
	\begin{enumerate}
		\item A \textit{dg functor} $F\colon\cC\to\cD$ between the dg categories $\cC$ and $\cD$ is a functor such that the map
		\[F\colon\hom^*(A,B)\to\hom^*(F(A),F(B))\]
		is a chain map for any $A, B \in \cC$, i.e.\ it preserves the degree of morphisms, and
		\[dF=Fd .\]
		
		\item A dg functor between the dg algebras is called a \textit{dga morphism}.		

	\end{enumerate}
\end{dfn}
A dg functor $F\colon\cC\to\cD$ induces a graded $k$-linear functor $H^*F\colon H^*\cC\to H^*\cD$, and a $k$-linear functor $H^0F\colon H^0\cC\to H^0\cD$.

\begin{dfn}
	\textit{The category of (small) dg categories} $\dgCat$ is the category whose objects are (small) dg categories and whose morphisms are dg functors, with the obvious identity morphism and the composition.
\end{dfn}
The category $\dgCat$ has the initial object, the empty dg category, and the final object, the category with one object $C$ where $\hom^*(C,C)$ is the zero module.

The following are important classes of morphisms in $\dgCat$.

\begin{dfn}
	Let $F\colon\cC\to\cD$ be a dg functor. 
	\begin{enumerate}
		\item $F$ is called \textit{quasi-fully faithful}, if the induced chain map
		\[F\colon\hom^*(A,B)\to\hom^*(F(A),F(B))\]
		is a quasi-isomorphism of chain complexes for any $A,B\in\cC$.
		\item F is called \textit{quasi-essentially surjective}, if the induced functor
		\[H^0F\colon H^0\cC\to H^0\cD\]
		is essentially surjective.
		\item $F$ is called a \textit{quasi-equivalence} if it is quasi-fully faithful and quasi-essentially surjective.
	\end{enumerate}  
\end{dfn}

\begin{rmk}
	A quasi-equivalence does not need to be invertible in $\dgCat$. We will be mostly concerned with dg categories up to quasi-equivalence, hence we want to invert them. In other words, we want to localise $\dgCat$ at quasi-equivalences. That is the main reason that we study the Dwyer-Kan model structure on $\dgCat$ in the next chapter.
\end{rmk}

\subsection{Dwyer-Kan Model Structure on $\dgCat$}\label{sec:dwyer-kan}

The following model structure on $\dgCat$ is due to Tabuada.

\begin{prp}[\cite{tabuada-model}]
	The category of dg categories $\dgCat$ has a cofibrantly generated model structure $\dgqe$, called Dwyer-Kan model structure, whose weak equivalences are quasi-equivalences, and whose fibrations are given by the isofibrations which are surjective on morphism complexes. Every object of $\dgqe$ is fibrant within this model structure.
\end{prp}

\begin{rmk}
	If we work with $\Z/2$-graded version of $\dgCat$, i.e.\ if the dg categories are $\Z/2$-graded, then it has the same model structure as above, as shown in \cite{dyckerhoff-kapranov}.
\end{rmk}

\begin{rmk}
	There are more general ways of obtaining this model structure on $\dgCat$. It can be carried from the model structure of category of chain complexes. See \cite{hinich} for dg algebras (or more generally, dg operads), and see \cite{lurie-topos}, \cite{berger-moerdijk}, and \cite{muro} for Dwyer-Kan model structure on enriched categories.
\end{rmk}

\begin{rmk}
	When we say that two dg categories $\cC$ and $\cD$ are quasi-equivalent (or are the same up to quasi-equivalence), we will mean that they are weakly equivalent in $\dgqe$, which means they are isomorphic in $\Ho(\dgqe)$ (see Definition \ref{dfn:homotopy-category}). In particular, if $\cC$ and $\cD$ are quasi-equivalent, we may not have a quasi-equivalence from $\cC$ to $\cD$, or $\cD$ to $\cC$.
\end{rmk}

For our purposes, it is crucial to understand cofibrant objects and cofibrations in $\dgqe$. For that, we will first define semifree dg categories (see \cite{drinfeld}) and semifree extensions.

\begin{dfn}
	\mbox{}
	\begin{enumerate}
		\item By \textit{adding a set of objects $\{C_i\}$ disjointly to a dg category $\cC$} we get a new dg category $\cC_o$ such that
		\begin{align*}
			\Ob\,\cC_o&:=\Ob\,\cC\cup\{C_i\},\\
			\hom^*_{\cC_o}(A,B)&:=\begin{cases}
				\hom^*_{\cC}(A,B), & \text{if }A,B\in\cC,\\
				k\langle 1_{A}\rangle, & \text{if }A=B\in\{C_i\},\\
				0, & \text{if }A\neq B\text{ and either }A\text{ or }B\text{ is in }\{C_i\} .
			\end{cases}
		\end{align*}
		We will sometimes write $\cC\cup\{C_i\}$ for $\cC_o$.
		\item 	By \textit{adding a set of (homogeneous) morphisms $\{f_i\colon B_i\to A_i\}$ (with prespecified gradings and differentials) semifreely to a dg category $\cC$} we get a new category $\cC_m$ such that
		\begin{align*}
			\Ob\,\cC_m&:=\Ob\,\cC,\\
			\hom^*_{\cC_m}(A,B)&:=\bigoplus_{n\geq 0}\bigoplus_{i_1,\dots,i_n}\hom^*_{\cC}(A_{i_n},B)\otimes k\langle f_{i_n}\rangle\otimes\ldots\\
			&\hspace{3cm}\ldots\otimes k\langle f_{i_2}\rangle\otimes\hom^*_{\cC}(A_{i_1},B_{i_2})\otimes k\langle f_{i_1}\rangle\otimes\hom^*_{\cC}(A,B_{i_1}),
		\end{align*}
		i.e.\ the morphisms $\{f_i\}$ are added to $\cC$ freely in algebra level. Compositions are given by concatenations, and compositions of $f_i$ are free. Gradings are obviously determined. Since $df_i$ are given, the rest of the differentials are given by the Leibniz rule.
		
		We will sometimes write $\cC\cup\{f_i\}$ for $\cC_m$.
	\end{enumerate}
\end{dfn}

\begin{dfn}
	\mbox{}
	\begin{enumerate}
		\item A dg functor $F\colon\cC\to\cD$ is a \textit{semifree extension}, if $F$ is an inclusion (i.e.\ it is a faithful functor and injective on objects), and if there is a filtration (indexed by an ordinal) \[F(\cC)\subset\cD_0\subset\cD_1\subset\ldots\subset\cD\] of dg categories such that 
		\begin{itemize}
			\item $\cD_0$ is obtained by adding objects disjointly to $F(\cC)$,
			\item $\cD_{j+1}$ is obtained from $\cD_j$ by adding homogeneous morphisms $f_i$'s semifreely such that $df_i$ is a morphism in $\cD_j$,
			\item if $\lambda$ is a limit ordinal, then $\cD_{\lambda}$ is obtained from $\displaystyle\lim_{\substack{\longrightarrow \\ j<\lambda}}\cD_j$ by adding homogeneous morphisms $f_i$'s semifreely such that $df_i$ is a morphism in $\displaystyle\lim_{\substack{\longrightarrow \\ j<\lambda}}\cD_j$.
		\end{itemize}
		\item We also call $\cD$ a \textit{semifree extension} of $\cC$ by the relevant objects and morphisms.
	\end{enumerate}
\end{dfn}

\begin{dfn}\label{dfn:semifree-dg-category}
	\mbox{}
	\begin{enumerate}
		\item A dg category $\cD$ is called \textit{semifree dg category} if $\cD$ is a semifree extension of the empty category. Explicitly, $\cD$ has a filtration (indexed by an ordinal) \[\cD_{\obj}:=\cD_0\subset\cD_1\subset\cD_2\subset\ldots\subset\cD\] of dg categories such that 
		\begin{itemize}
			\item $\cD_{\obj}$ is obtained by adding objects disjointly to the empty category,
			\item $\cD_{j+1}$ is obtained from $\cD_j$ by adding homogeneous morphisms $f_i$'s semifreely such that $df_i$ is a morphism in $\cD_j$,
			\item if $\lambda$ is a limit ordinal, then $\cD_{\lambda}$ is obtained from $\displaystyle\lim_{\substack{\longrightarrow \\ j<\lambda}}\cD_j$ by adding homogeneous morphisms $f_i$'s semifreely such that $df_i$ is a morphism in $\displaystyle\lim_{\substack{\longrightarrow \\ j<\lambda}}\cD_j$.
		\end{itemize}
		\item We call this filtration a \textit{semifree filtration} for $\cD$.
		\item $f_i$'s are called the \textit{generating morphisms} of $\cD$.
	\end{enumerate}
\end{dfn}

\begin{rmk}
	The underlying algebraic structure which is obtained by forgetting differential and grading from a semifree dg category is free, but not all such dg categories are semifree. Also, a semifree extension of a semifree dg category is again a semifree dg category.
\end{rmk}

We have a simple description for pushouts in $\dgCat$ when one of the arrows is a semifree extension. For general colimits (coequalisers) in the category of enriched categories, see \cite{wolff}. Also, compare with \cite{bednarczyk} where an explicit description of the coequalisers in the category of categories is given.

\begin{prp}\label{prp:colimit}
	Let $\alpha\colon\cC\to\cA$ be the semifree extension of a dg category $\cC$ by a set of objects $R$ and a set of generating morphisms $S$. The cobase change of $\alpha$ along a dg functor $\beta\colon\cC\to\cB$, i.e.\ the colimit
	\[\colim\left(
	\begin{tikzcd}
		\cA & & \cB\\
		& \cC\ar[lu,"\alpha"]\ar[ru,"\beta"']
	\end{tikzcd}
	\right)\]
	in $\dgCat$, is the semifree extension of $\cB$ by $R$ and $S'$, where $S'$ is a set of morphisms defined as follows: Let $\cC'$ (resp.\ $\cB'$) be the semifree extension of $\cC$ (resp.\ $\cB$) by the objects $R$. Let $\beta'\colon\cC'\to\cB'$ be the induced functor. Then
	\[S':=\{f'\colon \beta'(A)\to\beta'(B)\vb f\colon A\to B\text{ in }S\}\]
	is a set of morphisms defined for $\cB'$. Extend $\beta'$ to $\cC'\cup S\to\cB\cup S'$ by $\beta'(f):=f'\in S'$ for each $f\in S$. Then we define the gradings and differentials by $|f'|:=|f|$ and $df':=\beta'(df)$.
\end{prp}

\begin{proof}
	By the assumption, $\cA=\cC\cup R\cup S$. For any dg category $\cD$ and dg functors $F$ and $G$ satisfying $F\circ\alpha=G\circ\beta$, consider the following diagram:
	\[\begin{tikzcd}
		& \cD\\
		& \cB\cup R\cup S'\ar[u,"u",dashed]\\
		\cC\cup R\cup S\ar[ru,"\beta'"']\ar[ruu,"F"] & &\cB\ar[lu,"\alpha'"]\ar[luu,"G"']\\
		& \cC\ar[lu,"\alpha"]\ar[ru,"\beta"']
	\end{tikzcd}\]
	Here, $\alpha'\colon\cB\to\cB\cup R\cup S'$ is the semifree extension. We want to show that there exists a dg functor $u$ making this diagram commute, which will prove that $\cB\cup R\cup S'$ is the colimit.
	
	Clearly, $\beta'\circ\alpha=\alpha'\circ\beta$. Define the functor $u\colon \cB\cup R\cup S'\to\cD$ such that $u=G$ on $\cB$, $u=F$ on $R$, and
	\[u(f'):=F(f)\]
	for $f'\in S'$, where $f\in S$ is the corresponding morphism. To see that $u$ is a dg functor, we need to show that it commutes with differential. Without loss of generality, assume $df'\in\cB$. Then
	\[u(df')=G(df')=G\circ\beta(df)=F\circ\alpha(df)=F(df)=d(F(f))=d(u(f')) .\]
	It is easy to check that the dg functor $u$ makes the above diagram commute, hence $\cB\cup R\cup S'$ is indeed the colimit.
\end{proof}

\begin{prp}\label{prp:cofibration}
	The cofibrations in $\dgqe$ are retracts (in the arrow category) of semifree extensions, and the cofibrant objects in $\dgqe$ are retracts of semifree dg categories.
\end{prp}

\begin{proof}
	Since $\dgqe$ is cofibrantly generated, any cofibration in $\dgCat$ is a retract of a transfinite composition of cobase changes of coproducts of generating cofibrations, where generating cofibrations for $\dgqe$ are described in \cite{tabuada-model}. It is easy to see that a cobase change of a generating cofibration is either a semifree extension by one object or by one morphism by Proposition \ref{prp:colimit}. Hence, transfinite compositions of those are exactly semifree extensions.
\end{proof}

\begin{rmk}
	Proposition \ref{prp:cofibration} proves a fact that a semifree extension is a cofibration, and a semifree dg category is a cofibrant object in $\dgCat$. This fact plays an important role in the current paper. 
\end{rmk}

\begin{rmk}
	For dg algebras, this statement also appears in \cite{hinich} and \cite{drinfeld}, where semifree dg algebras are called standard cofibrant dg algebras in the former.
\end{rmk}

\begin{rmk}
	Some authors call a dg category semifree, if its underlying algebra is free. This definition is not useful for our purposes, as they may not be a cofibrant object in $\dgqe$. The existence of a semifree filtration is crucial to be a cofibrant object.
\end{rmk}

Finally, we will recall the existence of internal Homs in $\Ho(\dgqe)$ and their description via $A_{\infty}$-functors. See \cite{seidel} for a review of $A_{\infty}$-categories.

\begin{prp}[\cite{toen-morita}]
	The homotopy category $\Ho(\dgqe)$ has a closed symmetric monoidal structure. In particular, $\Ho(\dgqe)$ has internal Homs. For any given dg categories $\cC$ and $\cD$, we will denote their internal Hom by $\RHom(\cC,\cD)\in\Ho(\dgqe)$.
\end{prp}

\begin{prp}[\cite{faonte}]\label{prp:mod-ainf}
	Let $k$ be a field of characteristic zero. For any dg categories $\cC$ and $\cD$, the internal Hom $\RHom(\cC,\cD)$ is naturally isomorphic in $\Ho(\dgqe)$ to the dg category of $A_{\infty}$-functors, whose morphisms are $A_{\infty}$-natural transformations.
\end{prp}

\subsection{DG Localisation via Semifree Extension}

We will recall the definition of dg localisation, and express it as a semifree extension. This will be important when we define cylinder objects for semifree dg categories in the model categories of $\dgCat$.

\begin{dfn}
	Let $S$ be an arbitrary set of closed degree zero morphisms in a dg category $\cC$.
	A \textit{dg localisation} of $\cC$ at $S$ is the dg category $\cC[S^{-1}]$ with the dg functor $l\colon \cC\to \cC[S^{-1}]$ such that 
	\begin{itemize}
		\item for any dg category $\cD$, the induced morphism
		\[l^*\colon [\cC[S^{-1}],\cD]\to[\cC,\cD]\]
		in $\Ho(\dgqe)$ is injective, and
		\item the image of $l^*$ is a subset of morphisms $[\cC,\cD]$ consisting of all morphisms $F$ such that the induced functor $H^0F\colon H^0\cC\to H^0\cD$ sends each morphism in $S$ to an isomorphism in $H^0\cD$.
	\end{itemize} 
\end{dfn}

We will see that dg localisation exists, and it is unique up to quasi-equivalence.

\begin{dfn}
	We define $\cK_1$ to be the dg category with one object $C$ with $\hom^*(C,C)=k$, and $\cK_2$ to be the dg category with two objects $A$ and $B$ freely generated by a closed degree zero morphism $f\colon A\to B$.
\end{dfn}

\begin{prp}[\cite{toen-morita}]\label{prp:localisation}
	For any dg category $\cC$ and a subset $S$ of closed degree zero morphisms in $\cC$, the dg localisation $\cC[S^{-1}]$ exists and is unique up to quasi-equivalence, and it is given by
	\[\cC[S^{-1}]\simeq\hocolim\left(\begin{tikzcd}
		\coprod_{f\in S}\cK_1^f & & \cC\\
		& \coprod_{f\in S}\cK_2^f\ar[lu,"\alpha"]\ar[ru,"\beta"']
	\end{tikzcd}\right)\]
	in $\dgqe$,  where $\cK_1^f=\cK_1$ and $\cK_2^f =\cK_2$ for all $f\in S$, $f$ in $\cK_2^f$ is mapped to the identity morphism in $\cK_1^f$ under $\alpha$, and  to $f$ in $\cC$ under $\beta$.
\end{prp}

The following fact can be found in \cite{drinfeld}.

\begin{lem}\label{lem:semifree-trivial}
	The dg category $\cK_1$ is quasi-equivalent to the semifree dg category $\bar \cK_1$ with two objects $A$ and $B$, and generating morphisms $f,f',\hat f,\check f,\bar f$
	\[\begin{tikzcd}
		A\ar[loop left,"\hat f"]\ar[r,"f"]\ar[r,"\bar f",bend left=60] & B\ar[l,"f'",bend left=30]\ar[loop right,"\check f"]
	\end{tikzcd}\]
	with the gradings
	\[|f|=|f'|=0,\qquad |\hat f|=|\check f|=-1,\qquad |\bar f|=-2,\]
	and with the differentials
	\[df=df'=0,\qquad d\hat f=1_A-f'f,\qquad d\check f=1_B-f f',\qquad d\bar f=f\hat f - \check f f .\]
\end{lem}

Finally, we can express dg localisation as a semifree extension.

\begin{prp}\label{prp:localisation-semifree}
	Let $\cC$ be a semifree dg category. Then the dg localisation $\cC[S^{-1}]$ of $\cC$ can be expressed as the semifree extension of $\cC$ by the morphisms $f'_i,\hat f_i,\check f_i,\bar f_i$
	\[\begin{tikzcd}
		A_i\ar[loop left,"\hat f_i"]\ar[r,"f_i"]\ar[r,"\bar f_i",bend left=60] & B_i\ar[l,"f'_i",bend left=30]\ar[loop right,"\check f_i"]
	\end{tikzcd}\]
	for each $f_i\in\Hom^0(A_i,B_i)$ in $S$, with the gradings
	\[|f_i'|=0,\qquad |\hat f_i|=|\check f_i|=-1,\qquad |\bar f_i|=-2,\]
	and with the differentials
	\[df_i'=0,\qquad d\hat f_i=1_{A_i}-f_i'f_i,\qquad d\check f_i=1_{B_i}-f_if_i',\qquad d\bar f_i=f_i\hat f_i - \check f_if_i .\]
	In particular, the localisation functor $l\colon \cC\to \cC[S^{-1}]$ here is a semifree extension of $\cC$, hence a cofibration in $\dgqe$.
\end{prp}

\begin{proof}
	By Lemma \ref{lem:semifree-trivial}, we can replace $\cK_1$ with $\bar \cK_1$ in Proposition \ref{prp:localisation}, and we get
	\[\cC[S^{-1}]\simeq\hocolim\left(\begin{tikzcd}
		\coprod_{f\in S}\bar \cK_1^f & & \cC\\
		& \coprod_{f\in S}\cK_2^f\ar[lu,"\alpha"]\ar[ru,"\beta"']
	\end{tikzcd}\right),\]
	where $\alpha$ is now an inclusion which sends $f$ in $\cK_2^f$ to $f$ in $\bar \cK_1^f$. Hence, $\alpha$ is a cofibration. Also, all objects are cofibrant in the diagram above. Hence, by Proposition \ref{prp:model-hocolim-1}, the homotopy colimit becomes colimit. Then, Proposition \ref{prp:colimit} concludes the proof.
	
\end{proof}

\subsection{Quasi-equiconic and Morita Model Structure on $\dgCat$}

We studied the Dywer-Kan model structure on $\dgCat$ in Section \ref{sec:dwyer-kan}. There are two other model structures on $\dgCat$ we want to study in this paper, which are quasi-equiconic and Morita model structures. They are especially useful if one works with (idempotent-complete) pretriangulated dg categories.

\begin{dfn}
	We define $\Mod\, k$ to be the dg category of $\Z$-graded (co)chain complexes of $k$-modules, localised at quasi-isomorphisms. Equivalently, it is the dg category of cofibrant complexes of $k$-modules.
\end{dfn}

\begin{rmk}
	If $k$ is a field, every complex is cofibrant, hence no localisation is needed.
\end{rmk}

\begin{dfn}
	Let $\cC$ be a dg category. $\Mod\,\cC$ is defined as the dg category of \textit{$\cC$-modules}, which is given by the internal Hom $\RHom(\cC^{\op},\Mod\, k)$ in $\Ho(\dgqe)$.
\end{dfn}
 When $k$ is a field of characteristic zero, $\Mod\,\cC$ is equivalently characterised as the dg category of $A_{\infty}$-functors from $\cC^{\op}$ to $\Mod\,k$ by Proposition \ref{prp:mod-ainf}.
 	
\begin{dfn}
	\mbox{}
	\begin{enumerate}
		\item $\Tw\,\cC$ is defined as the dg category of \textit{twisted complexes} in $\cC$, or equivalently, the pretriangulated envelope of $\cC$.
		\item $\Perf\,\cC$ is the dg category of \textit{perfect $\cC$-modules}, which is the split-closure of $\Tw\,\cC$ (also called Karoubi envelope or idempotent completion of $\Tw\,\cC$). 
	\end{enumerate}
\end{dfn}
See \cite{seidel} for more explanation.
Note that $\Pi(\Tw\,\cC)$ is another notation for $\Perf\,\cC$.

\begin{prp}
	We have the quasi-fully faithful dg functor
	\begin{align*}
		\cY\colon \cC&\to \Mod\,\cC,\\
		C&\mapsto\hom^*(\uline,C),
	\end{align*}
	called \textit{dg Yoneda embedding}. Then, the followings hold; 
	\begin{itemize}
		\item $\Tw\,\cC\simeq\Tw(\cY(\cC))$ is the full dg subcategory of $\Mod\,\cC$ consisting of iterated mapping cones of the morphisms in $\cY(\cC)$.
		\item $\Perf\,\cC\simeq\Perf(\cY(\cC))$ is the full dg subcategory of $\Mod\,\cC$ obtained from $\Tw(\cY(\cC))$ by splitting direct summands.
		\item $\Mod\,\cC$ is the ind-completion of $\Perf(\cY(\cC))$, hence the (categorically) compact objects in $\Mod\,\cC$ are exactly the objects of $\Perf(\cY(\cC))$.
	\end{itemize} 
\end{prp}

There are two other model structures on $\dgCat$ which focus on twisted complexes and perfect modules. In order to define them, we need the following notion of equivalences.

\begin{dfn}
	Let $F\colon \cC\to\cD$ be a dg functor. 
	\begin{enumerate}
		\item $F$ is called \textit{quasi-equiconic} or a \textit{pretriangulated equivalence}, if the induced functor
		\[\Tw\,F\colon \Tw\,\cC\to\Tw\,\cD\]
		is a quasi-equivalence.
		\item $F$ is called a \textit{Morita equivalence}, if the induced functor
		\[\Perf\,F\colon \Perf\,\cC\to\Perf\,\cD\]
		is a quasi-equivalence.
	\end{enumerate}
\end{dfn}

The following model structures are described in \cite{tabuada-morita}.
	
\begin{prp}
	The category of dg categories $\dgCat$ has two other model structures which are obtained by left Bousfield localisations of the Dwyer-Kan model structure $\dgqe$:
	\begin{itemize}
		\item Quasi-equiconic model structure $\dgeq$, where the weak equivalences are pretriangulated equivalences. Fibrant objects in $\dgeq$ are \textit{pretriangulated dg categories}, i.e.\ dg categories of the form $\Tw\,\cC$ for some dg category $\cC$.
		
		\item Morita model structure $\dgmo$, where the weak equivalences are Morita equivalences. Fibrant objects in $\dgmo$ are \textit{idempotent-complete pretriangulated dg categories}, i.e.\ dg categories of the form $\Perf\,\cC$ for some dg category $\cC$.
	\end{itemize}
	In particular, they have the same cofibrations with the Dwyer-Kan model structure $\dgqe$ (see Definition \ref{dfn:bousfield}), and their homotopy categories are the full subcategories of $\Ho(\dgqe)$ by the respective fibrant objects.
\end{prp}

\begin{rmk}
	In particular, any semifree extension (resp.\ semifree dg category) is a cofibration (resp.\ cofibrant object) in Dwyer-Kan, quasi-equiconic, and Morita model structures on $\dgCat$.
\end{rmk}

\begin{rmk}	
	When we say that two dg categories $\cC$ and $\cD$ are pretriangulated equivalent (resp.\ Morita equivalent), or are the same up to pretriangulated equivalence (resp.\ Morita equivalence), we will mean that they are weakly equivalent in $\dgeq$ (resp.\ $\dgmo$), which means they are isomorphic in $\Ho(\dgeq)$ (resp.\ $\Ho(\dgmo)$). See Definition \ref{dfn:homotopy-category}.
\end{rmk}

We have two notions for generation of dg categories.

\begin{dfn}
	Let $\cC$ be the full dg subcategory of a dg category $\cD$. We say the objects of $\cC$ \textit{generate} (resp.\ \textit{split-generate}) $\cD$, if $\cC$ and $\cD$ are pretriangulated equivalent (resp.\ Morita equivalent).
\end{dfn}

\begin{rmk}
	If $\cC$ and $\cD$ are pretriangulated equivalent, then they are also Morita equivalent. Hence, if the objects of $\cC$ generate $\cD$, then they also split-generate $\cD$.
\end{rmk}

Finally, we will state the following fact.

\begin{prp}[\cite{cohn}]
	Underlying $\infty$-category of the Morita model structure $\dgmo$ is equivalent to the $\infty$-category of small idempotent-complete $k$-linear stable $\infty$-categories.
\end{prp}

\subsection{Cylinder Object for Semifree DG Categories}

The existence of (``very good'') cylinder objects (see Definition \ref{dfn:cylinder}) in $\dgCat$ is guaranteed by model category axioms. Here, we will explicitly describe an easily computable cylinder object for any semifree dg category (for the Dwyer-Kan, quasi-equiconic, and Morita model structures on $\dgCat$). This description of cylinder object will be useful as it is not much more complicated than the dg category it is associated to.

This cylinder object will be our main tool and it has many applications such as determining the (left) homotopic dg functors and homotopy colimit calculations. Next chapter will focus on the latter application for the chains on the loop spaces of lens spaces.

Let $\cC$ be a dg category. Let $\cC_i$ be a copy of $\cC$ for $i=1,2$. We will write $A_i\in\cC_i$ for the corresponding object to $A\in\cC$. Similarly, we will write $f_i\in\hom^*(A_i,B_i)$ for the corresponding morphism to $f\in\hom^*(A,B)$.

\begin{dfn}\label{dfn:cylinder-pre-dg}
	Let $\cC$ be a semifree dg category, and let $\cC_1, \cC_2$ be two copies of $\cC$. We define $\Cyl_0(\cC)$ to be the semifree extension of $\cC_1\amalg\cC_2$ by the morphisms consisting of
	\begin{itemize}
		\item closed degree zero morphisms $t_A\colon A_1\to A_2$ for each $A\in\cC$,
		\item degree $|f|-1$ morphisms $t_f\colon A_1\to B_2$ for each generating morphism $f\in\hom^*_{\cC}(A,B)$.
	\end{itemize}
	The differential of the added morphisms $t_f$ are given by the following; 
	if 
	\begin{gather*}
		df = \left\{\begin{matrix}
			c 1_A + \sum_{i=1}^m c_i f^{i,n_i}\ldots f^{i,j}\ldots f^{i,1}, \text{  if  } f \in \hom^*_{\cC}(A,A),\\ 
			\sum_{i=1}^m c_i f^{i,n_i}\ldots f^{i,j}\ldots f^{i,1}, \text{  if  } f \in \hom^*_{\cC}(A,B) \text{  such that  } A \neq B,
		\end{matrix}\right.
	\end{gather*}
	where $f^{i,j}$ are generating morphisms of $\cC$, and where $c, c_i\in k$, then
	\[dt_f=(-1)^{|f|}(f_2 t_A-t_B f_1)+\sum_{i=1}^m c_i \sum_{j=1}^{n_i}(-1)^{|f^{i,j-1}|+\ldots+|f^{i,1}|} f^{i,n_i}_2 \ldots f^{i,j+1}_2 t_{f^{i,j}} f^{i,j-1}_1\ldots f^{i,1}_1.\]		
\end{dfn}

\begin{rmk}
	It is straightforward to check that $d\circ d=0$, hence $\Cyl_0(\cC)$ is well-defined. Alternatively, one can observe that the collection of $t$-morphisms, i.e.\, 
	\[\{t_A, t_f \vb A \in \cC, f \text{  is a generating morphism in  } \cC \}\] look like components of an $A_{\infty}$-natural transformation $T$ (see \cite{seidel} for the definition) between dg endofunctors on $\cC$, with the property $T^n=0$ for $n\geq 2$. Their comparison also confirms that $d\circ d=0$. This is not a coincidence: An $A_{\infty}$-natural transformation $T$ between dg functors from a semifree dg category can be assumed to satisfy $T^n=0$ for $n\geq 2$. This is a work in progress of the first author.
\end{rmk}
	
\begin{rmk}
	$\Cyl_0(\cC)$ is indeed a semifree extension of $\cC\amalg\cC$ via the semifree filtration
	\begin{multline*}
		\cC_1\amalg\cC_2\subset(\cC_1\amalg\cC_2)\cup\{t_A\vb A\in\cC\}\subset(\cC_1\amalg\cC_2)\cup\{t_A\}\cup\{t_f\}\subset\\
		(\cC_1\amalg\cC_2)\cup\{t_A\}\cup\{t_f,t_g\}\subset\ldots\subset\Cyl_0(\cC)
	\end{multline*}
	coming from the semifree filtration of $\cC$
	\[\emptyset\subset\cC_{\obj}\subset\cC_{\obj}\cup\{f\}\subset\cC_{\obj}\cup\{f,g\}\subset\ldots\subset\cC,\]
	where $\cC_{\obj}$ has the same objects as $\cC$, and where all the morphisms in $\cC_{\obj}$ are generated by the identity morphisms.
\end{rmk}

\begin{dfn}\label{dfn:cylinder-dg}
	Let $\cC$ be a semifree dg category. We define $\Cyl(\cC)$ as the localisation $\Cyl_0(\cC)[S^{-1}]$ where $S=\{t_C\vb C\in\cC\}$. Equivalently, by Proposition \ref{prp:localisation-semifree}, $\Cyl(\cC)$ is the semifree extension of $\Cyl_0(\cC)$ by the morphisms $t'_C,\hat t_C,\check t_C,\bar t_C$ for each $C\in\cC$, with the gradings
	\[|t'_C|=0,\qquad |\hat t_C|=|\check t_C|=-1,\qquad |\bar t_C|=-2,\]
	and with the differentials
	\[dt'_C=0,\qquad d\hat t_C=1-t'_C t_C,\qquad d\check t_i=1-t_C t'_C,\qquad d\bar t_C=t_C\hat t_C - \check t_C t_C .\]
\end{dfn}

\begin{rmk}
	$\cC\amalg\cC\to\Cyl(\cC)$ is a semifree extension since $\cC\amalg\cC\to\Cyl_0(\cC)$ and the localisation are both semifree extensions. Moreover, $\Cyl(\cC)$ is a semifree dg category since $\cC\amalg\cC$ is a semifree dg category.
\end{rmk}

The following is one of our main theorems.

\begin{thm}\label{thm:cylinder}
	If $\cC$ is a semifree dg category, then $\Cyl(\cC)$ is a cylinder object for $\cC$ in the Dwyer-Kan, quasi-equiconic, and Morita model structures on the category of dg categories $\dgCat$. Furthermore, this cylinder object is ``very good'' in the sense of Remark \ref{rmk:very-good-cyl}.
\end{thm}

\begin{proof}
	Let $i\colon \cC\amalg\cC \to \Cyl(\cC)$ be the semifree extension described above. 
	It is easy to define the dg functor $p : \Cyl(\cC) \to \cC$ so that 
	\begin{itemize}
		\item $p\circ i$ is the codiagonal map, and
		\item for any $C\in\cC$ and any generating morphism $f$ in $\cC$,
			\begin{gather*}
			p(t_C)=p(t_C')=1_C,\\
			p(\hat t_C)=p(\check t_C)=p(\bar t_C)=0,\\
			p(t_f)=0.
		\end{gather*}
	\end{itemize}
	
	The dg functor $i$ is cofibration in any mentioned model categories since it is a semifree extension. We only need to show that $p$ is a weak equivalence in any models of $\dgCat$. It is enough to show that $p$ is a quasi-equivalence since every quasi-equivalence is a weak-equivalence in any models. For that, we will first show that $\Cyl(\cC)$ is quasi-equivalent to $\cC$.
	
	First, we can use the invertible morphisms $t_C$ to identify $C_1$ and $C_2$ in $\Cyl(\cC)$ for each $C\in\cC$. This makes $C_1=C_2$, $t_C=t'_C=1$, and $\hat t_C=\check t_C=\bar t_C=0$ for each $C\in\cC$. Hence, $\Cyl(\cC)$ is quasi-equivalent to the semifree extension of $\cC_1=\cC$ by the morphisms $f_2$ and $t_f$ for each generating morphism $f$ in $\cC$, call it $\cC'$. Note that we now have 
	\[dt_f=(-1)^{|f|}(f_2 - f_1)+\sum_{i=1}^m c_i \sum_{j=1}^{n_i}(-1)^{|f^{i,j-1}|+\ldots+|f^{i,1}|} f^{i,n_i}_2 \ldots f^{i,j+1}_2 t_{f^{i,j}} f^{i,j-1}_1\ldots f^{i,1}_1,\]
	if
	\[df=\sum_{i=1}^m c_i f^{i,n_i}\ldots f^{i,j}\ldots f^{i,1} .\]
	There is a semifree filtration
	\[\cC_1\subset\cC_1\cup\{f_2\}\subset\cC_1\cup\{f_2,t_f\}\subset\cC_1\cup\{f_2,t_f,g_2\}\subset\cC_1\cup\{f_2,t_f,g_2,t_g\}\subset\ldots\subset\cC',\]
	coming from the semifree filtration of $\cC$
	\[\emptyset\subset\cC_{\obj}\subset\cC_{\obj}\cup\{f\}\subset\cC_{\obj}\cup\{f,g\}\subset\ldots\subset\cC .\]
	Using this filtration, we can give an elementary automorphism of $\cC'$. Recall from \cite[Section 2.6]{subcritical} that an elementary automorphism of a semifree dg category can be thought as a change of variables
	\[v\mapsto u v + w\]
	respecting a filtration for any generating morphism $v$, where $u$ is a unit in $k$, and $w$ is a morphism which lives in the lower part of the filtration. The elementary automorphism of $\cC'$ we want to consider is the identity on objects, and
	\begin{align*}
		f_1&\mapsto f_1,\\
		f_2&\mapsto (-1)^{|f|}f_2 - (-1)^{|f|} f_1+\sum_{i=1}^m c_i \sum_{j=1}^{n_i}(-1)^{|f^{i,j-1}|+\ldots+|f^{i,1}|} f^{i,n_i}_2 \ldots f^{i,j+1}_2 t_{f^{i,j}} f^{i,j-1}_1\ldots f^{i,1}_1,\\
		t_f&\mapsto t_f,
	\end{align*}
	for any generating morphism $f$ in $\cC$. Note that $(-1)^{|f|}$ is a unit in $k$, and note that the term after $(-1)^{|f|}f_2$ lives in the lower part of the semifree filtration of $\cC'$, hence this is indeed an elementary automorphism of $\cC'$. Redefine $f_1,f_2,t_f$ using this automorphism, then we have
	\[df_2=0,\quad\text{and}\quad dt_f=f_2,\]
	for any generating morphism $f$ in $\cC$. These are  the only differential relations in $\cC'$ containing $f_2$ or $t_f$, hence by destabilisation (see \cite[Section 2.6]{subcritical}), $f_2$ and $t_f$ cancel each other. In other words, $\cC'$ is quasi-equivalent to $\cC_1=\cC$. This shows that $\Cyl(\cC)$ is quasi-equivalent to $\cC$.
	
	Finally, $p\colon\Cyl(\cC)\to\cC$ is a quasi-equivalence, because $p$ sends $f_1$ in $\Cyl(\cC)$ to $f$ in $\cC$ for each generating morphism $f$ in $\cC$. Hence, $\Cyl(\cC)$ is a cylindrical object for $\cC$ in any mentioned model categories.
	
	Also, it is easy to see that $p$ is a fibration, hence $\Cyl(\cC)$ is ``very good'' in the sense of Definition \ref{dfn:cylinder}.
\end{proof}

There is a useful generalisation of Theorem \ref{thm:cylinder} when a dg category is expressed as a localisation.

\begin{thm}\label{thm:cylinder-loc}
	If $\cC$ is a semifree dg category, and if $S$ is a subset of closed degree zero morphisms in $\cC$, then $\Cyl(\cC)[\bar S^{-1}]$ is a ``very good'' cylindrical object for $\cC[S^{-1}]$ for the Dwyer-Kan, quasi-equiconic, and Morita model structures on the category of dg categories $\dgCat$, where
	\[\bar S=\{f_1,f_2\vb f\in S\},\]
	and where $f_1$ and $f_2$ are the copies of $f$ in $\cC_1\sqcup\cC_2\subset\Cyl(\cC)$ with $\cC_1=\cC_2=\cC$.
\end{thm}

\begin{proof}
	If we look at the decomposition of the codiagonal map
	\[\cC[S^{-1}]\amalg \cC[S^{-1}]\xrightarrow{i}\Cyl(\cC)[\bar S^{-1}]\xrightarrow{p}\cC[S^{-1}],\]
	it is easy to see that $i$ is a semifree extension, hence a cofibration. Moreover, $p$ is a quasi-equivalence since
	\[\Cyl(\cC)\to\cC\]
	is a quasi-equivalence. Also, $p$ is clearly a fibration.
\end{proof}

\subsection{Homotopy Colimit for Semifree DG Categories}

We will discuss homotopy pushouts in this section. Note that any colimit can be obtained from pushouts.

The following application of the cylinder object will be our main tool for homotopy colimit calculations.

\begin{thm}\label{thm:hocolim}
	Let $\cA, \cB, \cC$ be semifree dg categories, $\alpha\colon\cC\to\cA$ and $\beta\colon\cC\to\cB$ be dg functors. The homotopy colimit
	\[\hocolim\left(
	\begin{tikzcd}
		\cA & & \cB\\
		& \cC\ar[lu,"\alpha"]\ar[ru,"\beta"']
	\end{tikzcd}
	\right),\]
	for the Dwyer-Kan, quasi-equiconic, and Morita model structures on the category of dg categories $\dgCat$ can be expressed as the semifree dg category $\cD[\{t_C\vb C\in\cC\}^{-1}]$, where $\cD$ is a semifree extension of $\cA\amalg\cB$ by the morphisms consisting of
	\begin{itemize}
		\item closed degree zero morphisms $t_C\colon \alpha(C)\to \beta(C)$ for each $C\in\cC$,
		\item degree $|f|-1$ morphisms $t_f\colon \alpha(A)\to \beta(B)$ for each generating morphism $f\in\hom^*_{\cC}(A,B)$.
	\end{itemize}
	The differential of the added morphisms are given by the following; 
	if 
	\begin{gather*}
	df = \begin{cases}
		c 1_A + \sum_{i=1}^m c_i f^{i,n_i}\ldots f^{i,j}\ldots f^{i,1}, &\text{  if  } f \in \hom^*_{\cC}(A,A),\\ 
		\sum_{i=1}^m c_i f^{i,n_i}\ldots f^{i,j}\ldots f^{i,1}, &\text{  if  } f \in \hom^*_{\cC}(A,B) \text{  such that  } A \neq B,
		\end{cases}
	\end{gather*}
	where $f^{i,j}$ are generating morphisms of $\cC$, and where $c, c_i\in k$, then 
	\begin{multline*}
	dt_f=(-1)^{|f|}(\beta(f) t_{A}-t_{B} \alpha(f))+\\
	\sum_{i=1}^m c_i \sum_{j=1}^{n_i}(-1)^{|f^{i,j-1}|+\ldots+|f^{i,1}|} \beta(f^{i,n_i}) \ldots \beta(f^{i,j+1}) t_{f^{i,j}} \alpha(f^{i,j-1})\ldots \alpha(f^{i,1}).
	\end{multline*}
\end{thm}

\begin{proof}
	It directly follows from Proposition \ref{prp:hocolim-to-colim}, Theorem \ref{thm:cylinder}, and Proposition \ref{prp:colimit}.
\end{proof}

\begin{rmk}
	Theorem \ref{thm:hocolim} describes homotopy colimit of semifree dg categories as a semifree dg category since $\cD$ is a semifree dg category and a localisation of a semifree dg category can be expressed as a semifree dg category by Proposition \ref{prp:localisation-semifree}. This is useful in many ways. To name one, it allows using the homotopy colimit formula successively to compute homotopy colimit of larger diagrams.
\end{rmk}

There is an easier homotopy colimit formula when the dg categories are given by localisations. It is very useful, for example, when calculating wrapped Fukaya category of Weinstein manifolds via gluing partial wrapped Fukaya categories of Weinstein sectors (see \cite{gps2}). See also Chapter \ref{sec:wrap-lens} for an application.

\begin{thm}\label{thm:hocolim-loc}
	Let $\cA, \cB, \cC$ be semifree dg categories, $\alpha\colon\cC\to\cA$ and $\beta\colon\cC\to\cB$ be dg functors. If $S_{\cA},S_{\cB},S_{\cC}$ are subsets of closed degree zero morphisms in $\cA$, $\cB$, $\cC$, respectively, such that the morphisms in $\alpha(S_{\cC})$ (resp.\ $\beta(S_{\cC})$) are invertible in $H^0(\cA[S^{-1}_{\cA}])$ (resp.\ $H^0(\cB[S_{\cB}^{-1}])$), then we have
	\[\hocolim\left(
	\begin{tikzcd}[column sep=0.1cm]
		\cA[S_{\cA}^{-1}] & & \cB[S_{\cB}^{-1}]\\
		& \cC[S_{\cC}^{-1}]\ar[lu,"\alpha"]\ar[ru,"\beta"']
	\end{tikzcd}
	\right)
	\simeq
	\hocolim\left(
	\begin{tikzcd}
		\cA & & \cB\\
		& \cC\ar[lu,"\alpha"]\ar[ru,"\beta"']
	\end{tikzcd}
	\right)[S^{-1}],
	\]
	up to weak equivalence in the Dwyer-Kan, quasi-equiconic, and Morita model structures on the category of dg categories $\dgCat$, where
	\[S=S_{\cA}\sqcup S_{\cB},\]
	and where the formula for the latter homotopy colimit is given in Theorem \ref{thm:hocolim}.
\end{thm}

\begin{proof}
	The former homotopy colimit is quasi-equivalent to
	\[\colim\left(
	\begin{tikzcd}[column sep=0.05cm]
		\cA[S_{\cA}^{-1}] & & \Cyl(\cC)[\bar S_{\cC}^{-1}] & & \cB[S_{\cB}^{-1}]\\
		& \cC[S_{\cC}^{-1}]\ar[lu,"\alpha"]\ar[ru,"i_1"'] & & \cC[S_{\cC}^{-1}]\ar[lu,"i_2"]\ar[ru,"\beta"'] 
	\end{tikzcd}
	\right),
	\]
	by Proposition \ref{prp:hocolim-to-colim} and Theorem \ref{thm:cylinder-loc}. Then, Proposition \ref{prp:colimit} and Theorem \ref{thm:hocolim} conclude the proof.
\end{proof}

\subsection{Homotopy Limit for Representations of Semifree DG Categories}

Lastly, we will state a theorem, which is useful, for example, when we glue microlocal sheaf categories, see \cite{pinwheel}. Microlocal sheaf categories are naturally presented as $\Mod\,\cC$ where $\cC$ is a semifree dg category.

\begin{thm}\label{thm:holim}
	Let $\cA, \cB, \cC$ be semifree dg categories, $\alpha\colon\cC\to\cA$ and $\beta\colon\cC\to\cB$ be dg functors. Then we have
	\[\holim\left(
	\begin{tikzcd}[column sep=0.1cm]
		\Mod\,\cA\ar[rd,"\Mod\,\alpha"'] & & \Mod\,\cB\ar[ld,"\Mod\,\beta"]\\
		& \Mod\,\cC
	\end{tikzcd}
	\right)
	\simeq
	\Mod\left(\hocolim\left(
	\begin{tikzcd}
		\cA & & \cB\\
		& \cC\ar[lu,"\alpha"]\ar[ru,"\beta"']
	\end{tikzcd}
	\right)\right),\]
	up to quasi-equivalence, where $\Mod\,\alpha$ and $\Mod\,\beta$ are induced functors, and the formula for the latter homotopy colimit is given in Theorem \ref{thm:hocolim}.
\end{thm}

\begin{proof}
	Since $\RHom(\uline,\uline)$ is internal Hom in $\Ho(\dgqe)$, the functor
	\begin{align*}
		\Mod\colon \Ho(\dgqe)^{\op}&\longrightarrow\Ho(\dgqe)\\
		\cC &\longmapsto\Mod(\cC):=\RHom(\cC^{\op},\Mod\,k)
	\end{align*}
	preserves limits. The rest is Theorem \ref{thm:hocolim}.
\end{proof}

\begin{rmk}
	Compare Theorem \ref{thm:holim} with the homotopy limit formula given in \cite{canonaco}. The former gives an easier way to compute homotopy limits when it applies.
\end{rmk}

\section{Wrapped Fukaya Category of the Cotangent Bundles of Lens Spaces}\label{sec:wrap-lens}

In this chapter, we will describe the wrapped Fukaya category of the cotangent bundles of the lens spaces (Theorem \ref{thm:wrap-lens}), and also the chains on the based loop space of lens spaces (Theorem \ref{thm:loop-lens}). For that, we will use the Heegaard diagram for lens spaces. This will induce homotopy colimit diagrams for the mentioned invariants for lens spaces by Theorem \ref{thm:wrapped-hocolim} and Theorem \ref{thm:loop-hocolim}, which will allow us to apply our homotopy colimit formulas (Theorem \ref{thm:hocolim} and Theorem \ref{thm:hocolim-loc}) to calculate them.

\subsection{Wrapped Fukaya Categories and Chains on the Based Loops Spaces}

We will recall two classes of symplectic manifolds: Liouville and Weinstein manifolds. We will also recall a powerful invariant associated to them: Wrapped Fukaya category. Our main references are \cite{weinstein}, \cite{gps1}, and \cite{gps2}. Let $k$ be a commutative ring.

\begin{dfn}
	A \textit{Liouville manifold} $(W,\theta,Z)$ is an (even-dimensional) smooth manifold  $W$, with a 1-form $\theta$ on $W$, called a \textit{Liouville form}, and a complete vector field $Z$ on $W$, called a \textit{Liouville vector field}, such that
	\begin{itemize}
		\item $d\theta$ is a symplectic form on $W$,		
		\item the equation $d\theta(Z,\uline)=\theta$ holds,
		\item there is an exhaustion $W_1\subset W_2\subset\ldots\subset W$ such that each $W_i$ is a compact domain with a smooth boundary,
		\item $Z$ is outwardly transverse to $\partial W_i$ for each $i$.
	\end{itemize}
\end{dfn}

\begin{dfn}
	The \textit{skeleton} of a Liouville manifold $(W,\theta,Z)$ is given by
	\[\bigcup_{i=1}^\infty\bigcap_{t> 0} Z^{-t}(W_i),\]
	where
	\[Z^{-t}\colon W\to W\]
	is the negative flow of $Z$, and $W_1\subset W_2\subset\ldots\subset W$ is a compact exhaustion of $W$.
\end{dfn}

\begin{dfn}
	A Liouville manifold $(W,\theta,Z)$ is \textit{finite-type}, if its skeleton is compact. In this case, we can write
	\[W = W_c \cup_{\partial W_c} (\partial W_c\times [0,\infty)),\]
	where $W_c$ is a compact domain with a smooth boundary containing the skeleton, and $Z$ is outwardly transverse to $\partial W_c$. $\partial W_c\times [0,\infty)$ is called the \textit{cylindrical end} of $W$. The Liouville form $\theta$ becomes
	\[e^r\theta|_{\partial W_C}\]
	on the cylindrical end, where $r$ is the radial coordinate on $[0,\infty)$.
\end{dfn}

From now on, we will assume that all Liouville manifolds are finite-type.

\begin{dfn}
	An \textit{exact Lagrangian} $L$ in a Liouville manifold $(W,\theta,Z)$ is a half-dimensional submanifold of $W$ such that
	\[\theta|_L=df\]
	for some smooth function $f\colon L\to \R$. $L$ is \textit{with cylindrical end}, if $\theta|_L=0$ outside a compact domain, or equivalently, it is of the form
	\[L=L_c\cup_{\partial L_c}(\partial L_c\times [0,\infty)),\]
	where $\partial L_c\times [0,\infty)$ lives inside the cylindrical end of $W$ nicely.
\end{dfn}

\begin{dfn}
	A \textit{Weinstein manifold} $(W,\theta,Z)$ is a Liouville manifold $(W,\theta,Z)$ such that $Z$ is gradient-like (see \cite{weinstein}) with respect to a proper Morse function $\phi\colon W\to \R_{\geq 0}$.
\end{dfn}

\begin{exm}
	The cotangent bundle $T^*M$ of a closed smooth manifold $M$ is a Weinstein manifold with the skeleton $M$.
\end{exm}

\begin{prp}
	For a $2n$-dimensional Weinstein manifold $W$, the Morse function $\phi$ gives a handle decompositon for $W$ with Weinstein handles whose cores are isotropic submanifolds of $W$. In particular, the indices of the handles are at most $n$.
\end{prp}

\begin{dfn}
	A Weinstein $n$-handle (resp.\ ``$<n$''-handle) of a $2n$-dimensional Weinstein manifold $W$ is called a \textit{critical handle} (resp \textit{subcritical handle}) of $W$.
\end{dfn}

\begin{prp}
	The skeleton of a Weinstein manifold $W$ is the union of cores of the Weinstein handles. In particular, the skeleton is isotropic (possibly singular) in $W$.
\end{prp}

\begin{dfn}
	A \textit{Liouville sector} (resp.\ \textit{Weinstein sector}) is a Liouville (resp.\ Weinstein) manifold with boundary satisfying the conditions given in \cite{gps1} near its boundary.
\end{dfn}

Studying Lagrangians in Liouville sectors gives a powerful invariant, presented as an $A_{\infty}$-category . Note that any $A_{\infty}$-category is equivalent to a dg category. See \cite{seidel} for a review of $A_{\infty}$-categories.

\begin{dfn}\label{dfn:wrapped-fukaya}
	The \textit{wrapped Fukaya category} $\cW(W)$ of a Liouville sector $W$ is an $A_{\infty}$-category  whose objects are exact immersed Lagrangians with cylindrical end, and morphisms are given by the direct limit
	\[\hom^*(L_1,L_2):=\lim_{\substack{\longrightarrow \\ L_1\to L_1^+}}k\langle L_1^+\cap L_2\rangle,\]
	over positive isotopies $L_1\to L_1^+$, see \cite{gps2} for the details. $A_{\infty}$-relations are coming from counting pseudo-holomorphic polygons in $W$ bounded by Lagrangians.
	
	Note that $\cW(W)$ can be always made $\Z/2$-graded, however, it can be made $\Z$-graded only when $2c_1(W)=0$ in $H^2(W;\Z)$. The definition of $\Z$-graded $\cW(W)$ depends on the choice of a quadratic complex volume form on $TW$ as explained in \cite{seidel}. This gives effectively $H^1(W;\Z)$ many choices for $\cW(W)$. For cotangent bundles, there is a canonical choice.
\end{dfn}

There is a generation result for wrapped Fukaya categories by \cite{cdrgg} and \cite{gps2}. For cotangent bundles, this is originally due to \cite{abouzaid-wrapped-generation}.

\begin{thm}\label{thm:wrapped-generation}
	The wrapped Fukaya category $\cW(W)$ of a Weinstein manifold $W$ is generated by the Lagrangian cocores of the critical handles of $W$. In particular, $\cW(T^*M)$ is generated by a cotangent fibre of $T^*M$ when $M$ is a connected closed smooth manifold.
\end{thm}

The following cosheaf property will be our main tool when calculating wrapped Fukaya categories.

\begin{thm}[\cite{gps2}]\label{thm:wrapped-hocolim}
	Let $W=W_1\cup W_2$ be a Liouville manifold such that $W_1$ and $W_2$ are Weinstein sectors meeting along a hypersurface in $W$. If the neighborhood of the hypersurface is $F\times T^*[0,1]$ where $F$ is a Weinstein sector up to a deformation, then we have
	\[\cW(W)\simeq\hocolim\left(\begin{tikzcd}[column sep=0.1cm]
		\cW(W_1) & & \cW(W_2)\\
		& \cW(F)\ar[lu]\ar[ru]
	\end{tikzcd}\right),\]
	up to pretriangulated equivalence.
\end{thm}

\begin{rmk}
	There is a procedure called ``arborealisation'' that transforms a Weinstein manifold into one which is a union of Weinstein sectors whose skeleta are ``arboreal singularities'' (see \cite{arboreal} for the definition, and see \cite{arboreal-starkston} and \cite{arboreal-alvarez} for the cases where this procedure applies). The wrapped Fukaya category of such Weinstein sectors is combinatorially described in \cite{arboreal} as a semifree dg category. Hence, the difficulty of the calculation of wrapped Fukaya categories reduces to the problem of taking homotopy colimit of semifree dg categories via Theorem \ref{thm:wrapped-hocolim}.
	
	This is exactly the case our Theorem \ref{thm:hocolim} deals with, if the arrows between semifree dg categories in the diagram of Theorem \ref{thm:wrapped-hocolim} can be described. This can be done relatively easily at least for cotangent bundles and plumbings with clean intersections (see \cite{abouzaid-plumbing}). Next sections will focus on the calculation of the wrapped Fukaya category of the cotangent bundle of lens spaces via this approach. The study of the wrapped Fukaya category of the plumbings with clean intersections is work in progress of the authors of this paper.
\end{rmk}

\begin{rmk}
	One can also consider the dg category of (unbounded) microlocal sheaves $\mSh(W)$ on the skeleton of a Weinstein manifold $W$ (see \cite{wrapped} and \cite{gps3} for the definition) which is quasi-equivalent to $\Mod(\cW(W))$. The full dg subcategory of $\Mod(\cW(W))$ consisting of (categorically) compact objects is Morita equivalent to $\cW(W)$. $\mSh(W)$ has a sheaf property (in contrast to Theorem \ref{thm:wrapped-hocolim} where we have a cosheaf property), hence we have
	\[\mSh(W)\simeq\holim\left(\begin{tikzcd}[column sep=0.1cm]
		\mSh(W_1)\ar[rd] & & \mSh(W_2)\ar[ld]\\
		& \mSh(F)
	\end{tikzcd}\right),\]
	up to quasi-equivalence. See \cite{pinwheel} for an application of this approach to calculation of wrapped Fukaya category of some rational homology balls. Our Theorem \ref{thm:holim} proposes a simpler calculation of this type of homotopy limits when it applies.
\end{rmk}

Let us focus on the Weinstein manifolds that are cotangent bundles from now on. There is a relation between the wrapped Fukaya category and the loop space of the base manifold in this case.

\begin{thm}[\cite{abouzaid-loops}]\label{thm:wrapped-loops}
	Let $M$ be a connected closed smooth manifold, and let $x\in M$ be a point. The endomorphism algebra of the cotangent fibre $T_x^*M$ is the dg algebra which is the singular chain complex of the based (Moore) loop space $\Omega_x M$ of $M$, denoted by $C_{-*}(\Omega_xM)$, where the product is Pontryagin product with the concatenations of loops. In particular, we have
	\[\cW(W)\simeq C_{-*}(\Omega_xM),\]
	up to pretriangulated equivalence by Theorem \ref{thm:wrapped-generation}.
\end{thm}

\begin{rmk}
	One can also consider the based (not Moore!) loop space $\Omega_xM$ as an $A_{\infty}$-space where $A_{\infty}$-operations are coming from the concatenation of loops and (higher) homotopies between them. The resulting $A_{\infty}$-algebra $C_{-*}(\Omega_xM)$ is quasi-equivalent to the dg algebra $C_{-*}(\Omega_xM)$ in Theorem \ref{thm:wrapped-loops}.
\end{rmk}

Theorem \ref{thm:wrapped-hocolim} also induces a gluing property for the chains on the based loop spaces.

\begin{thm}\label{thm:loop-hocolim}
	Let $M=M_1\cup M_2$ be a connected smooth manifold such that $M_1$, $M_2$, and $M_1\cap M_2$ are connected smooth manifolds and open in $M$. Let $x\in M_1\cap M_2$ be a point. Then we have
	\[C_{-*}(\Omega_x M)\simeq\hocolim\left(\begin{tikzcd}[column sep=0.1cm]
		C_{-*}(\Omega_x M_1) & & C_{-*}(\Omega_x M_2)\\
		& C_{-*}(\Omega_x (M_1\cap M_2))\ar[lu]\ar[ru]
	\end{tikzcd}\right),\]
	up to quasi-equivalence.
\end{thm}

\begin{proof}
	Geometrically, Theorem \ref{thm:wrapped-hocolim} identifies $T_x^*M_1$ and $T_x^*M_2$ along the inclusion of $T_x^*(M_1\cap M_2)$ into both, to get $T_x^*M$. This induces a gluing (homotopy colimit) of endomorphism ($A_{\infty}$-)algebras of the cotangent fibres, which are given by the chains on the based loop spaces of the bases by Theorem \ref{thm:wrapped-loops}.
\end{proof}

\begin{rmk}
	Note that
	\[H_0(\Omega_x M)=\pi_1(M)\]
	by definition, hence the above theorem can be regarded as an extension of the Seifert-Van Kampen theorem.
\end{rmk}

\begin{rmk}
	Both $C_{-*}(\Omega_x M)$ and $\cW(T^*M)$ are invariants for the smooth manifold $M$, however, the former is a priori a stronger invariant since we consider it up to quasi-equivalence, whereas the latter is considered up to pretriangulated equivalence. 
\end{rmk}

\begin{rmk}
	If $M$ is simply connected, $C_{-*}(\Omega_x M)$ is determined by the singular cochain complex $C^*(M)$ of $M$ with the cup product (see \cite{ekholm-lekili} for a more general Koszul duality statement). Hence, this invariant is only interesting when $M$ is not simply connected. We will see that when $M$ is a lens space, $C_{-*}(\Omega_x M)$ is a strictly stronger invariant than $C^*(M)$ with cup product and $\pi_1(M)$ since it distinguishes the homotopy type of lens spaces where the others cannot.
\end{rmk}

The homology of $C_{-*}(\Omega_x M)$ also gives an invariant, although it is determined by the fundamental group and the universal cover of $M$.

\begin{prp}\label{prp:homology-universal}
	Let $\widetilde M$ is the universal cover of a connected topological space $M$. Then we have the homotopy equivalence
	\[\Omega_x M\simeq \bigsqcup_{\gamma\in\pi_1(M)}\Omega_x \widetilde M .\]
	In particular, we have the isomorphism
	\[H_{-*}(\Omega_x M)\simeq \bigoplus_{\gamma\in\pi_1(M)}H_{-*}(\Omega_x \widetilde M)\]
	of graded $k$-modules.
\end{prp}

\begin{proof}
	The based loop space $\Omega_xM$ has $H_0(\Omega_xM)=\pi_1(M)$-many connected components. Label the connected component corresponding $\gamma\in\pi_1(M)$ by $(\Omega_xM)_{\gamma}$. Connected components are homotopy equivalent via
	\begin{align*}
		(\Omega_xM)_{\gamma}&\longrightarrow(\Omega_xM)_{\gamma'},\\
		\eta&\longmapsto \gamma'\circ\eta\circ\gamma^{-1},		
	\end{align*}
	where the operation is the concatenation of loops. In particular, when $\gamma=1$ is the identity, we have the homotopy equivalence
	\[(\Omega_xM)_1\simeq \Omega_x\widetilde M\]
	since any element (loop) in $(\Omega_xM)_1$ becomes an element (loop) in $\Omega_x\widetilde M$ after lifting. This concludes the proof.
\end{proof}

\begin{rmk}
	Since $\widetilde M$ is simply connected, $C_{-*}(\Omega_x \widetilde M)$ is determined by the cohomology of $\widetilde M$ with the cup product. Hence, $H_{-*}(\Omega_x M)$ is not a strong invariant for $M$. Therefore, the dg structure of $C_{-*}(\Omega_x M)$ (which comes from $A_{\infty}$-structure of $\Omega_x M$, not just from its homotopy type) is important to get a strong invariant for $M$ as we will see in the next sections for the lens spaces.
\end{rmk}

\subsection{Heegaard Diagram for 3-Manifolds and Lens Spaces}

Before studying the lens spaces, we will recall Heegaard diagrams for 3-manifolds. A reference for this section is \cite{heegaard}.

Any oriented closed 3-manifold $M$ admits a Heegaard diagram $(\Sigma_g,\alpha_1,\ldots,\alpha_g,\beta_1,\ldots,\beta_g)$ where
\begin{itemize}
	\item $\Sigma_g$ is the genus $g$ closed topological surface,
	
	\item $\alpha_i$ and $\beta_i$ are closed embedded curves in $\Sigma_g$,
	
	\item $\alpha$-curves are disjoint, and $\beta$-curves are disjoint,
	
	\item $[\alpha_i]$ are linearly independent in $H_1(\Sigma_g;\Z)$, and $[\beta_i]$ are linearly independent in $H_1(\Sigma_g;\Z)$.
\end{itemize}
A Heegaard diagram determines a closed 3-manifold by gluing two genus $g$ handlebodies $U_1$ and $U_2$ along their common boundary $\Sigma_g$, where $U_1$ is obtained from $\Sigma_g$ by attaching $g$ many $D^2\times D^1$ along their attaching circles $\partial D^2\simeq \alpha_i$ and then attaching $D^3$ along its boundary $S^2$. $U_2$ is obtained similarly from $\beta$-curves.

In this section, we will focus on the case where 3-manifolds are given by a Heegaard diagram with $g=1$. If $M$ is given by the Heegaard diagram $(\Sigma_1,\alpha,\beta)$ for some $\alpha$ and $\beta$, then
\begin{equation}\label{eq:colim-3-mfd}
	M\simeq \colim\left(
	\begin{tikzcd}
		U_1 & & U_2\\
		& \Sigma_1\ar[lu,"i_1"]\ar[ru,"i_2"']
	\end{tikzcd}
	\right)
\end{equation}
as a topological space, where $\Sigma_1$ is the torus, $U_1$ and $U_2$ are solid tori. The maps $i_j$ are determined by the induced map on the first homology when $g=1$. So, the only data is the kernel of the maps
\[(i_j)_*\colon H_1(\Sigma_1;\Z)=\Z\oplus\Z\to\Z=H_1(U_j;\Z)\]
for $j=1,2$. Note that $[\alpha]=0$ in $H_1(U_1;\Z)$, and $[\beta]=0$ in $H_1(U_2;\Z)$. We can always choose a basis for $m,n\in H_1(\Sigma_1;\Z)=\Z\oplus\Z$ such that $[\alpha]=n$ in $H_1(\Sigma_1;\Z)$.

\begin{prp}\label{prp:heegaard-genus-1}
	A Heegaard diagram $(\Sigma_1,\alpha,\beta)$ with $[\alpha]=n$ determines one of the following:
	\begin{itemize}
		\item $S^3$, if $[\beta]=m$,
		
		\item $S^1\times S^2$, if $[\beta]=n$,
		
		\item lens spaces $L(p,q)$ for $p>q\geq 1$ and $(p,q)=1$, if $[\beta]=pm+qn$.
	\end{itemize}
\end{prp}

\begin{rmk}
	The lens space $L(2,1)$ is homeomorphic to $\RP^3$.
\end{rmk}

\begin{prp}
	Every topological 3-manifold has a unique smooth structure. Moreover, the smooth structure of topological 3-manifolds are uniquely determined by their simple homotopy type.
\end{prp}

Finally, we will recall some facts regarding lens spaces.

\begin{prp}\label{prp:lens-homology}
	For a lens space $L(p,q)$, we have the homology groups
	\[H_n(L(p,q);\Z)=\begin{cases}
		\Z, & \text{if }n=0,\\
		\Z_p, & \text{if }n=1,\\
		0, & \text{if }n=2,\\
		\Z, & \text{if }n=3,
	\end{cases}\]
	the cohomology groups
	\[H^n(L(p,q);\Z)=\begin{cases}
		\Z, & \text{if }n=0,\\
		0, & \text{if }n=1,\\
		\Z_p, & \text{if }n=2,\\
		\Z, & \text{if }n=3,
	\end{cases}\]
	and the homotopy groups
	\[\pi_n(L(p,q))=\begin{cases}
		\Z_p, & \text{if }n=1,\\
		\pi_n(S^3), & \text{if }n\geq 2.
	\end{cases}\]
	Moreover, the cup product on $H^n(L(p,q);\Z)$ does not depend on $q$.
\end{prp}

\begin{thm}[\cite{reidemeister-lens}, \cite{brody-lens}]\label{thm:lens-classification}
	The lens spaces $L(p,q)$ and $L(p',q')$ are homotopy equivalent if and only if $p'=p$ and
	\[q'\equiv \pm\alpha^2 q \quad \textup{(mod $p$)}\]
	for some $\alpha\in\Z$. $L(p,q)$ and $L(p',q')$ are simple homotopy equivalent if and only if $p'=p$ and
	\[q'\equiv \pm q^{\pm 1}\quad \textup{(mod $p$)} .\]
\end{thm}

\begin{rmk}
	The homology/cohomology groups (with cup product) and the homotopy groups do not distinguish two non-homotopic lens spaces $L(p,q)$ and $L(p,q')$. 
	
	The classification up to homotopy is given by the torsion linking form
	\[L\colon \text{Tor}(H_1(L(p,q);\Z))\otimes \text{Tor}(H_1(L(p,q);\Z))\to\Q/\Z,\]
	where $\text{Tor}(H_1(L(p,q);\Z))$ is the torsion part of the group $H_1(L(p,q);\Z)$.
	
	The classification up to simple homotopy is given by the Reidemeister torsion.
	
	For $s\in L(p,q)$, we will show that the chains $C_{-*}(\Omega_sL(p,q))$ on the based loop space of $L(p,q)$ also classifies lens spaces up to homotopy equivalence, and $C_{-*}(\Omega_sL(p,q))$ does not detect the simple homotopy type of lens spaces in Part \ref{part dga computations}.
\end{rmk}

\subsection{Wrapped Fukaya Category of the Cotangent Bundle of Torus}

To understand the wrapped Fukaya category of the cotangent bundles of 3-manifolds coming from genus 1 Heegaard diagrams, we need first to figure out the wrapped Fukaya category of the cotangent bundle of torus, $\cW(T^*\Sigma_1)$.

There are multiple ways to describe $\cW(T^*\Sigma_1)$. It is pretriangulated equivalent to
\[\cW(T^*S^1)\otimes\cW(T^*S^1)\]
by \cite{gps2}. It can be also calculated using Circle Lemma in \cite{pinwheel}. Here, we will calculate it using Theorem \ref{thm:hocolim-loc} we proved, which will naturally describe it as a localisation of a semifree dg category.

We will write
\[k\langle x_1,x_2,\ldots,x_n\rangle\]
for the semifree dg algebra with the single object $L$ and the generating morphisms
\[x_1,x_2,\ldots,x_n\in\hom^*(L,L) .\]
Also, recall that we write
\[ \cC[\{y_1,y_2,\ldots,y_m\}^{-1}]\]
for the dg localisation of a dg category $\cC$ at a set of closed degree zero morphisms $\{y_1,y_2,\ldots,y_m\}$.

First note that for $p\in S^1$, we have
\[C_{-*}(\Omega_pS^1)\simeq k\langle x\rangle [\{x\}^{-1}],\]
up to quasi-equivalence, and
\[\cW(T^*S^1)\simeq k\langle x\rangle [\{x\}^{-1}],\]
up to pretriangulated equivalence, where $dx=0$ and $|x|=0$. In particular, we have \[\Hom^*(L_x,L_x)\simeq k[x,x^{-1}] .\]
for a cotangent fibre $L_x$ of $T^*S^1$. The computation can be found e.g. in \cite{pinwheel}.

Note also that
\[\cW(T^*(M\times I))\simeq \cW(T^*M),\]
up to pretriangulated equivalence, for any smooth manifold $M$ and interval $I$ (see e.g. \cite{gps2}). Then we have the following result:

\begin{prp}\label{prp:torus}
	For the torus $\Sigma_1$, we have
	\[\cW(T^*\Sigma_1)\simeq k\langle m,n,h\rangle[\{m,n\}^{-1}]\]
	up to pretriangulated equivalence, where $m,n$ are closed degree zero morphisms, and $h$ is a degree $-1$ morphism with the differential
	\[dh=mn-nm .\]
\end{prp}

\begin{proof}
	We can get a torus via gluing two cylinders along their boundaries, hence we have
	\[\cW(T^*\Sigma_1) \simeq \hocolim\left(
		\begin{tikzcd}[column sep=0.01cm]
			\cW(T^*(S^1\times I)) & & \cW(T^*(S^1\times I))\\
			& \cW(T^*(S^1\amalg S^1))\ar[lu]\ar[ru]
		\end{tikzcd}
	\right),\]
	up to pretriangulated equivalence by Theorem \ref{thm:wrapped-hocolim}. Then we have
	\[\cW(T^*\Sigma_1) \simeq \hocolim\left(
		\begin{tikzcd}
			k\langle x\rangle[\{x\}^{-1}] & & k\langle y\rangle[\{y\}^{-1}]\\
			& k\langle u\rangle[\{u\}^{-1}]\amalg k\langle v\rangle[\{v\}^{-1}]\ar[lu,"\alpha"]\ar[ru,"\beta"']
		\end{tikzcd}\right),\
	\]
	where
	\begin{align*}
		x&\in\hom^*(L_x,L_x),\\
		y&\in\hom^*(L_y,L_y),\\
		u&\in\hom^*(L_u,L_u),\\
		v&\in\hom^*(L_v,L_v)
	\end{align*}
	are closed degree zero morphisms, and
	\begin{align*}
		\alpha(u)&=\alpha(v)=x,\\
		\beta(u)&=\beta(v)=y .
	\end{align*}
	Then, Theorem \ref{thm:hocolim-loc} gives
	\[\cW(T^*\Sigma_1)\simeq \cD[\{x,y,t_{L_u},t_{L_v}\}^{-1}],\]
	up to pretriangulated equivalence, where $\cD$ is the semifree dg category with two objects $L_x$ and $L_y$ and with the generating morphisms
	\begin{align*}
		x &\in\hom^*(L_x,L_x),\\
		y &\in\hom^*(L_y,L_y),\\
		t_{L_u},t_{L_v},t_u,t_v &\in\hom^*(L_x,L_y),
	\end{align*}
	such that $|t_{L_u}|=|t_{L_v}|=|t_u|+1=|t_v|+1=0$, $dt_{L_u}=dt_{L_v}=0$, and
	\begin{align*}
		dt_u&=\beta(u) t_{L_u}-t_{L_u} \alpha(u)=y t_{L_u}-t_{L_u}x,\\
		dt_v&=\beta(v) t_{L_v}-t_{L_v} \alpha(v)= y t_{L_v} - t_{L_v}x .
	\end{align*}
	Use invertible $t_{L_v}$ to identify $L_x$ and $L_y$, and set $t_{L_v}=1$. Then we get
	\begin{align*}
		dt_u&=yt_{L_u}-t_{L_u}x,\\
		dt_v&=y-x .
	\end{align*}
	Using $t_v$, we can identify $y=x$ and set $t_v=0$. By the relabeling
	\[m:=x,\qquad n:=t_{L_u},\qquad h:=t_u,\]
	we get the result.
\end{proof}

\begin{rmk}
	As stated in Definition \ref{dfn:wrapped-fukaya}, there are $H^1(\Sigma_1;\Z)=\Z\oplus\Z$ many choices for $\cW(T^*\Sigma_1)$, and the one we described corresponds to the canonical choice. The other choices can be obtained by letting $m$ and $n$ to have an arbitrary degree.
\end{rmk}

Similarly, we can apply Theorem \ref{thm:loop-hocolim} and Theorem \ref{thm:hocolim-loc} to get the following proposition.

\begin{prp}
	For the torus $\Sigma_1$ and $p\in\Sigma_1$, we have
	\[C_{-*}(\Omega_p\Sigma_1)\simeq k\langle m,n,h\rangle[\{m,n\}^{-1}],\]
	up to quasi-equivalence, where $m,n$ are closed degree zero morphisms, and $h$ is a degree $-1$ morphism with the differential
	\[dh=mn-nm .\]
\end{prp}

\subsection{Wrapped Fukaya Category for Genus 1 Heegaard Diagrams}

The gluing diagram (\ref{eq:colim-3-mfd}) for a 3-manifold $M$ given by genus 1 Heegaard diagram induces a gluing diagram for the wrapped Fukaya category of $T^*M$ by Theorem \ref{thm:wrapped-hocolim}
\[
\cW(T^*M)\simeq \hocolim\left(
\begin{tikzcd}
	\cW(T^*U_1) & & \cW(T^*U_2)\\
	& \cW(T^*\Sigma_1)\ar[lu,"i_1"]\ar[ru,"i_2"']
\end{tikzcd}
\right),
\]
up to pretriangulated equivalence, where $U_i$ is a solid tori and $\Sigma_1$ is a torus. Here,
\[\cW(T^*U_i)\simeq\cW(T^*S^1)\simeq k\langle x\rangle[\{x\}^{-1}]\]
up to pretriangulated equivalence, where $dx=0$ and $|x|=0$. From Proposition \ref{prp:torus}, we also have 
\[\cW(T^*\Sigma_1)\simeq k\langle m,n,h\rangle[\{m,n\}^{-1}]\]
up to pretriangulated equivalence, where $m,n\in\hom^*(L,L)$ are closed degree zero morphisms, and $h\in\hom^*(L,L)$ is a degree $-1$ morphism with the differential
\[dh=mn-nm .\]
Then we get
\begin{equation}\label{eq:hocolim-wrapped-3-mfd}
	\cW(T^*M)\simeq \hocolim\left(
	\begin{tikzcd}[column sep=0.1cm]
		k\langle x_1\rangle[\{x_1\}^{-1}] & & k\langle x_2\rangle[\{x_2\}^{-1}]\\
		& k\langle m,n,h\rangle[\{m.n\}^{-1}]\ar[lu,"i_1"]\ar[ru,"i_2"'] .
	\end{tikzcd}
	\right),
\end{equation}
up to pretriangulated equivalence. We can assume that
\[i_1(m)=x_1,\qquad i_1(n)=1,\]
as in the classification given by Proposition \ref{prp:heegaard-genus-1}, and $i_2(m)$ and $i_2(n)$ are in the polynomial ring $k[x,x^{-1}]$ where we denote $x:=x_2$. Then we have the following proposition.

\begin{prp}\label{prp:wrapped-heegaard}
	If $M$ is a 3-manifold given by a genus 1 Heegaard diagram, and if we assume $i_2(m)$ and $i_2(n)$ in the diagram (\ref{eq:hocolim-wrapped-3-mfd}) are in the polynomial ring $k[x]$, then
	\[\cW(T^*M)\simeq k\langle x,y,z\rangle[\{x\}^{-1}],\]
	up to pretriangulated equivalence, where $|x|=0, |y|=-1$, $|z|=-2$, and
	\begin{align*}
		dx&=0,\\
		dy&=i_2(n)-1,\\
		dz&=i_2(m)y-y\, i_2(m) .
	\end{align*}
	The dg functor $i_2$ is determined by the Heegaard diagram of $M$.
\end{prp}

\begin{proof}
	We can assume that
	\[i_1(h)=i_2(h)=0\]
	as $\Hom^{-1}(L_j,L_j)\simeq 0$ in $k\langle x_j\rangle[\{x_j\}^{-1}]$ for $j=1,2$. Then, by the assumption, the dg functor
	\[i_2\colon k\langle m,n,h\rangle[\{m,n\}^{-1}]\to k\langle x_2\rangle[\{x_2\}^{-1}]\]
	descends to the dg functor
	\[i_2\colon k\langle m,n,h\rangle \to k\langle x_2\rangle .\]
	After recalling
	\[i_1(m)=x_1,\qquad i_1(n)=1,\]
	we can apply Theorem \ref{thm:hocolim-loc} and get
	\[\cW(T^*M)\simeq\cD[\{x_1,x_2,t_L\}^{-1}],\]
	up to pretriangulated equivalence, where $\cD$ is semifree dg category with two objects $L_1$ and $L_2$, and with the generating morphisms
	\begin{align*}
		x_1&\in\hom^*(L_1,L_1),\\
		x_2&\in\hom^*(L_2,L_2),\\
		t_L, t_m, t_n, t_h&\in\hom^*(L_1,L_2),
	\end{align*}
	such that $|t_L|=|t_m|+1=|t_n|+1=|t_h|+2=0$, and
	\begin{align*}
		dt_L&=0,\\
		dt_m&=i_2(m)t_L-t_L i_1(m)=i_2(m)t_L-t_L x_1,\\
		dt_n&=i_2(n)t_L-t_L i_1(n)=i_2(n) t_L-t_L,\\
		dt_h&=-(i_2(h)t_L-t_L i_1(h))+ i_2(m) t_n + t_m i_1(n) - i_2(n) t_m - t_n i_1(m),\\
		&=i_2(m) t_n + t_m - i_2(n) t_m -t_n x_1 .
	\end{align*}
	Use invertible $t_L$ to identify $L_1$ and $L_2$, and set $t_L=1$. Then
	\begin{align*}
		dt_m&=i_2(m)- x_1,\\
		dt_n&=i_2(n) -1,\\
		dt_h&=i_2(m) t_n-t_n x_1 + (1 - i_2(n)) t_m .
	\end{align*}
	Define $z:=t_h+t_nt_m$ to replace $t_h$, and we get 
	\begin{align*}
		dz&=i_2(m) t_n-t_n x_1 + (1 - i_2(n)) t_m + (i_2(n) -1)t_m-t_n(i_2(m)-x_1)\\
		&=i_2(m)t_n -t_n i_2(m) .
	\end{align*}
	We can use $dt_m=i_2(m)-x_1$ to set $x_1=i_2(m)$ and $t_m=0$. Then we get
	\[\cW(T^*M)\simeq k\langle x_2,t_n,z\rangle[\{i_2(m),x_2\}^{-1}],\]
	up to pretriangulated equivalence. However, $m$ is invertible, hence $i_2(m)$ is already invertible in $k\langle x_2\rangle[\{x_2\}^{-1}]$. This means that we don't need to invert $i_2(m)$. Then, we get the proposition by setting $x:=x_2$ and $y:=t_n$.
\end{proof}

Next, we will give a quick calculation for $\cW(T^*S^3)$ and $\cW(T^*(S^2\times S^1))$ using Proposition \ref{prp:wrapped-heegaard}, and leave lens spaces to the next section.

\begin{prp}\label{prp:wrap-s3}
	We have
	\[\cW(T^*S^3)\simeq k\langle z\rangle,\]
	up to pretriangulated equivalence, where $|z|=-2$ and $dz=0$.
\end{prp}

\begin{proof}
	For $S^3$, Proposition \ref{prp:heegaard-genus-1} suggests that $i_2(m)=1$, $i_2(n)=x$. Hence, by Proposition \ref{prp:wrapped-heegaard}, we get
	\[\cW(T^*S^3)\simeq k\langle x,y,z\rangle[\{x\}^{-1}]\]
	where $|x|=0, |y|=-1$, $|z|=-2$, and
	\[dx=0,\quad dy=x-1,\quad dz=y-y=0 .\]
	Using $dy=x-1$, we can set $x=1$ and $y=0$, which proves the proposition.
\end{proof}

Similarly, we can apply Theorem \ref{thm:loop-hocolim} and Theorem \ref{thm:hocolim-loc} to get the following proposition.

\begin{prp}\label{prp:s3-dga}
	If $p\in S^3$, we have
	\[C_{-*}(\Omega_p S^3)\simeq k\langle z\rangle,\]
	up to quasi-equivalence, where $|z|=-2$ and $dz=0$.
\end{prp}

\begin{prp}\label{prp:wrap-s1xs2}
	We have
	\[\cW(T^*(S^1\times S^2))\simeq k\langle x,y,z\rangle[\{x\}^{-1}],\]
	where $|x|=0, |y|=-1$, $|z|=-2$, and
	\[dx=0,\quad dy=0,\quad dz=xy-yx .\]
\end{prp}

\begin{proof}
	For $S^1\times S^2$, Proposition \ref{prp:heegaard-genus-1} suggests that $i_2(m)=x$, $i_2(n)=1$. Hence, Proposition \ref{prp:wrapped-heegaard} gives the result.
\end{proof}

\begin{rmk}
	Note that
	\[\cW(T^*(S^1\times S^2))\simeq\cW(T^*S^1)\otimes\cW(T^*S^2),\]
	up to pretriangulated equivalence. Then, one can interpret Proposition \ref{prp:wrap-s1xs2} such that $x$ is coming from $\cW(T^*S^1)\simeq k\langle x\rangle$, $y$ is coming from $\cW(T^*S^2)\simeq k\langle y\rangle$ (see \cite{pinwheel}), and $z$ commutes $x$ and $y$ because of the tensor product.
\end{rmk}

\subsection{Wrapped Fukaya Category of the Cotangent Bundles of Lens Spaces}

In this section, we will describe the wrapped Fukaya category of the cotangent bundle of a lens space $L(p,q)$ using Proposition \ref{prp:wrapped-heegaard}. To do that, consider the diagram (\ref{eq:hocolim-wrapped-3-mfd}) for the lens space $L(p,q)$
\begin{equation*}
	\cW(T^*L(p,q))\simeq \hocolim\left(
	\begin{tikzcd}[column sep=0.1cm]
		k\langle x_1\rangle[\{x_1\}^{-1}] & & k\langle x_2\rangle[\{x_2\}^{-1}]\\
		& k\langle m,n,h\rangle[\{m,n\}^{-1}]\ar[lu,"i_1"]\ar[ru,"i_2"'] .
	\end{tikzcd}
	\right),
\end{equation*}
where
\[i_1(m)=x_1,\qquad i_1(n)=1.\]
Let $x:=x_2$. Here, $i_2(m)$ and $i_2(n)$ are in the polynomial ring $k[x,x^{-1}]$. To apply Proposition \ref{prp:wrapped-heegaard}, we need to have
\[i_2(m), i_2(n)\in k[x] .\]
For this to be possible, we first need to change the Heegaard diagram for $L_{p,q}$ given in Proposition \ref{prp:heegaard-genus-1} in such a way that $[\beta]=pm-qn$. This can be done by replacing $n$ with $-n$. This imposes the requirement
\[i_2(m^p n^{-q})=1 .\]
Also, geometrically, the monodromy $x$ of the solid torus should be in the image of $i_2$. From these requirements, we need to deduce $i_2(m)$ and $i_2(n)$.

Since $p$ and $q$ are relatively prime, there exists $r,s\in\Z$ such that $pr+qs=1$. Consider the change of basis for the torus
\begin{align*}
	f\colon k\langle m,n,h\rangle[\{m,n\}^{-1}] &\rightarrow k\langle u,v,w\rangle[\{u,v\}^{-1}],\\
	m &\mapsto u^r v^q,\\
	n &\mapsto u^{-s} v^p,
\end{align*}
with the inverse
\begin{align*}
	f'\colon k\langle u,v,w\rangle[\{u,v\}^{-1}] &\rightarrow k\langle m,n,h\rangle[\{m,n\}^{-1}],\\
	u &\mapsto m^p n^{-q},\\
	v &\mapsto m^s n^r .
\end{align*}
Note that the image of $h$ and $w$ can be given but is not relevant here. We can write $i_2$ as
\[i_2\colon k\langle m,n,h\rangle[\{m,n\}^{-1}] \xrightarrow{f} k\langle u,v,w\rangle[\{u,v\}^{-1}]\xrightarrow{g}k\langle x\rangle[\{x\}^{-1}],\]
where $g$ is just the trivial inclusion of the torus to the solid torus as its boundary, with
\[g(u)=1,\qquad g(v)=x,\qquad g(w)=0 .\]
This satisfy the requirements above, i.e.\
\[i_2(m^p n^{-q})=1,\]
and $x$ is in the image of $i_2$. This shows that
\[i_2(m)=x^q,\quad i_2(n)=x^p .\]
Then we finally get one of our main theorems.

\begin{thm}\label{thm:wrap-lens}
	If $L(p,q)$ is a lens space with $p>q\geq 1$ and $(p,q)=1$, then
	\[\cW(T^*L(p,q))\simeq k\langle x,y,z\rangle,\]
	up to pretriangulated equivalence, where $|x|=0, |y|=-1$, $|z|=-2$, and
	
	\begin{align*}
		dx&=0,\\
		dy&=1-x^p,\\
		dz&=x^q y-yx^q .
	\end{align*}
	We will denote this differential graded algebra (dga) $k\langle x,y,z\rangle$ by $\cC_{p,q}$.
\end{thm}

\begin{proof}
	Since $i_2(m)=x^q$ and $i_2(n)=x^p$ for $L(p,q)$, Proposition \ref{prp:wrapped-heegaard} gives
	\[\cW(T^*L(p,q))\simeq k\langle x,y,z\rangle[\{x\}^{-1}],\]
	up to pretriangulated equivalence, with the given degrees and differentials above. Note that we do not need to invert $x$ because $dy=x^p-1$ already implies that $x$ is invertible.

	After replacing $y$ with $-y$ and $z$ with $-z$, we get the result.
\end{proof}

\begin{rmk}
	There is a unique choice for $\cW(T^*L(p,q))$ since $H^1(L(p,q);\Z)=0$, as stated in Definition \ref{dfn:wrapped-fukaya}.
\end{rmk}

Similarly, we can apply Theorem \ref{thm:loop-hocolim} and Theorem \ref{thm:hocolim-loc} to get the following theorem.

\begin{thm}\label{thm:loop-lens}
	If $s\in L(p,q)$, we have
	\[C_{-*}(\Omega_s L(p,q))\simeq k\langle x,y,z\rangle,\]
	up to quasi-equivalence, where $|x|=0, |y|=-1$, $|z|=-2$, and
		
	\begin{align*}
		dx&=0,\\
		dy&=1-x^p,\\
		dz&=x^q y-yx^q .
	\end{align*}
\end{thm}

One can easily see that
\[H^0(\cC_{p,q})= k[\Z_p]\]
from the description of $\cC_{p,q}=C_{-*}(\Omega_s L(p,q))$. We can confirm this also by
\[H^0(\cC_{p,q})=H_0(\Omega_s L(p,q))= k[\pi_1(L(p,q))]\]
and $\pi_1(L(p,q))=\Z_p$ from Proposition \ref{prp:lens-homology}. Moreover, we have the following.

\begin{prp}\label{prp:cpq-cohomology}
	If $s\in L(p,q)$, we have
	\[H^n(\cC_{p,q})=\begin{cases}
		k[\Z_p], & \text{if $n\leq 0$ and $n$ is even,}\\
		0, & \text{otherwise.}
	\end{cases}\]
\end{prp}

\begin{proof}
	Note that $S^3$ is the universal cover of the lens space $L(p,q)$. Since we have $\pi_1(L(p,q))=\Z_p$, by Proposition \ref{prp:homology-universal} we have the isomorphism
	\[H^*(\cC_{p,q})=H_{-*}(\Omega_s L(p,q))\simeq \bigoplus_{\gamma\in\Z_p}H_{-*}(\Omega_s S^3)\]
	of graded $k$-modules. By Proposition \ref{prp:s3-dga}, we have
	\[C_{-*}(\Omega_s S^3)\simeq k\langle z\rangle,\]
	up to quasi-equivalence, where $|z|=-2$ and $dz=0$. Then
	\[H_{-n}(\Omega_s S^3)=\begin{cases}
		k, & \text{if $n\leq 0$ and $n$ is even,}\\
		0, & \text{otherwise.}
	\end{cases}\]
	This concludes the proof.
\end{proof}

\begin{rmk}\label{rmk:cohomology-lens}
	Note that the cohomology $H^*(\cC_{p,q})$ distinguishes the lens spaces $L(p,q)$ and $L(p',q')$ when $p\neq p'$. However, when $p=p'$, $H^*(\cC_{p,q})$ is the same for both lens spaces (even with the product structure, as we will see in Part \ref{part dga computations}). Moreover, for degree reasons, (higher) Massey products $\mu^n$ with $n\geq 3$ on $H^*(\cC_{p,q})$ are zero when $n$ is odd. In Part \ref{part dga computations}, we will study the differential graded structure of $\cC_{p,q}$ to show that $\cC_{p,q}$ in fact detects the homotopy type of the lens spaces when one considers more than the cohomology of $\cC_{p,q}$.
\end{rmk}

\clearpage

\part{Analysis on DGA Invariants of Lens Spaces}
\label{part dga computations}

\section{Setting}
\label{section setting}

\subsection{Results in Part \ref{part dga computations}}
\label{subsection results in part 2}
In the previous part, we constructed a differential graded algebra (dga) $\mathcal{C}_{p,q}$ over a coefficient ring $k$, from a lens space $L(p,q)$, so that $\mathcal{C}_{p,q}$ is pretriangulated equivalent to the wrapped Fukaya category of the cotangent bundle of the lens space $L(p,q)$ (see Theorem \ref{thm:wrap-lens}).
In Part \ref{part dga computations}, we will perform some computations on $\mathcal{C}_{p,q}$ for the case of $k = \mathbb{Z}$. 
By doing this, we prove the following Theorems \ref{thm homotopy type to quasi} and \ref{thm quasi to homotopy type}.

\begin{rmk}
	We note that in the rest of this paper, we fix the coefficient ring $k$ as $\mathbb{Z}$ if we do not specify it. 
	The reason why we do not choose a field coefficient will be explained in Section \ref{subsect properties of chi}, in Remark \ref{rmk:lens-field}.
\end{rmk}

\begin{thm}
	\label{thm homotopy type to quasi}
	If $L(p_1,q_1)$ and $L(p_1,q_2)$ are of the same homotopy type, then $\mathcal{C}_{p_1,q_1}$ and $\mathcal{C}_{p_2,q_2}$ are quasi-equivalent.
\end{thm}

As a corollary of Theorem \ref{thm homotopy type to quasi}, we obtain Corollary \ref{cor Fukaya category}.

\begin{cor}
	\label{cor Fukaya category}
	The wrapped Fukaya category of $T^*L(p,q)$ is an invariant of the homotopy type of $L(p,q)$.
\end{cor}

\begin{rmk}\label{rmk:wrapped-lens}
	Corollary \ref{cor Fukaya category} gives an example of quasi-equivalence between (pretriangulated closures of) wrapped Fukaya categories of two Weinstein manifolds, which are not induced from a symplectomorphisms between them. 
	More precisely, by \cite{abouzaid-kragh}, it is known that $T^*L(p_1,q_1)$ and $T^*L(p_2,q_2)$ are symplectomorphic if and only if $L(p_1,q_1)$ and $L(p_2,q_2)$ are diffeomorphic. 
	Since there is a pair of lens spaces $\big(L(p_1,q_1), L(p_2,q_2)\big)$ such that not diffeomorphic but homotopic, Corollary \ref{cor Fukaya category} provides an example of a quasi-equivalence between wrapped Fukaya categories, which is not induced from a symplectomorphism. 
	
	From the view point of homological mirror symmetry, it seems that the above argument implies a restriction on the mirror side of a lens space $L(p_i,q_i)$ in the above pair. 
	This is because, by \cite[Thoerem 2.5]{bondal-orlov}, if $X$ is a smooth irreducible projective variety with ample canonical or anticanonical sheaf, and if $D^b_{coh}(X)$ is equivalent to $D^b_{coh}(X')$ for some other smooth algebraic variety $X'$, then $X$ and $X'$ are isomorphic to each other. 
	Roughly, every equivalence between $D^b_{coh}(X)$ and $D^b_{coh}(X')$ must be induced from a geometric equivalence. 
	Thus, we expect that the mirror of $T^*L(p_i,q_i)$ where $L(p_i,q_i)$ is in the above pair cannot be a smooth irreducible projective variety with ample canonical or anticanonical sheaf.
\end{rmk}

Theorem \ref{thm quasi to homotopy type} is the inverse directional statement of Theorem \ref{thm homotopy type to quasi}.

\begin{thm} 
	\label{thm quasi to homotopy type}
	If $\mathcal{C}_{p_1,q_1}$ and $\mathcal{C}_{p_2,q_2}$ are quasi-equivalent, then $L(p_1,q_1)$ and $L(p_2,q_2)$ are of the same homotopy type. 
\end{thm}

We note that by Theorem \ref{thm:lens-classification}, $L(p_1,q_1)$ and $L(p_2,q_2)$ are homotopic to each other if and only if 
\[p_1 = p_2, \text{  and  } bq_2 = a^2 q_1 + c p_1,\]
where $a, c \in \mathbb{Z}$, and where $b$ is either $1$ or $-1$.
Moreover, one can easily prove that if $\mathcal{C}_{p_1, q_1}$ and $\mathcal{C}_{p_2,q_2}$ are quasi-equivalent, then $p_1 = p_2$ by taking the cohomology, see Remark \ref{rmk:cohomology-lens}. 

From this, in the rest of the current paper, we simply use $p$ instead of $p_1$ or $p_2$ in Theorems \ref{thm homotopy type to quasi} and \ref{thm quasi to homotopy type}. 

\subsection{Notation}
\label{subsect notation}
We review some notions from the previous part, and partially set notation for Part \ref{part dga computations} in Section \ref{subsect notation}.

The differential graded algebra $\mathcal{C}^*_{p,q}$ is the semifree dga given in Theorem \ref{thm:wrap-lens}, i.e., 
\begin{gather*}
	\cC_{p,q} \simeq k\langle x,y,z\rangle\, \text{  such that,} \\
	|x|=0, |y|=-1, |z|=-2,\\
	dx =0 , dy = 1-x^p, z = x^q y - y x^q,
\end{gather*}
and let $H^*_{p,q}$ denote the cohomology of $\mathcal{C}_{p,q}$.
We note that $\mathcal{C}_{p,q}$ is determined from the lens space $L(p,q)$, thus, $p$ and $q$ are relatively primes.
It induces that there are $\bar{q}, r \in \mathbb{Z}$ such that 
\begin{gather}
	\label{eqn p,q}
	\bar{q} q = rp +1.
\end{gather}
We note that there are infinitely many $\bar{q}$ and $r$. 
Let $\bar{q}$ and $r$ be the smallest positive numbers satisfying Equation \eqref{eqn p,q}.

We fix two degree $-2$ elements $\chi, \Lambda \in \mathcal{C}^{-2}_{p,q}$ with the numbers $\bar{q}, r$ as follows:
\begin{gather}
	\label{eqn chi}
	\chi := y\big(\sum_{i=1}^qx^{p(i-1)}\big)y + \sum_{i=1}^p x^{q(p-i)}zx^{q(i-1)},\\
	\label{eqn Lambda}
	\Lambda := yx\big(\sum_{i=1}^{r} x^{p(i-1)}\big)y + \sum_{i=1}^{\bar{q}}x^{q(\bar{q}-i)}zx^{q(i-1)}.
\end{gather}
For simplicity, we set 
\begin{gather}
	\label{eqn f_n(x)}
	f_n(x):= 1+ x + x^2 + \cdots + x^{n-1}.
\end{gather}
Then, one obtains 
\begin{gather}
	\label{eqn def chi Lambda}
	\chi = yf_q(x^p)y + \sum_{i=1}^p x^{q(p-i)}zx^{q(i-1)}, \hspace{1em} \Lambda = yxf_{r}(x^p)y + \sum_{i=1}^{\bar{q}}x^{q(\bar{q}-i)}zx^{q(i-1)}.
\end{gather}

The followings are results of easy computations.
\begin{align}
	\label{eqn d chi}
	d \chi &= (1- x^p) f_q(x^p) y  - y f_q(x^p)(1-x^p) + \sum_{i=1}^p x^{q(p-i)} (x^q y - y x^q) x^{q(i-1)} \\ 
	\notag &= (1- x^{pq}) y - y (1- x^{pq}) + \sum_{i=1}^p (x^{q(p-i+1)} y x^{q(i-1)} - x^{q(p-i)} y x^{qi} ) \\
	\notag &= 0, \\
	\label{eqn d Lambda}
	d \Lambda &= (1-x^p) x f_r(x^p) y - yx f_r(x^p) (1-x^p) + \sum_{i=1}^{\bar{q}}x^{q(\bar{q}-i)}(x^qy-yx^q)x^{q(i-1)} \\
	\notag &= (1 - x^{pr}) xy - yx (1-x^{pr}) + \sum_{i=1}^{\bar{q}}(x^{q(\bar{q}-i+1)}yx^{q(i-1)} - x^{q(\bar{q}-i)}yx^{qi}) \\
	\notag &= xy - yx - x^{pr+1}y + yx^{pr+1} + x^{q\bar{q}}y - yx^{q\bar{q}} \\
	\notag &= xy-yx.
\end{align}

In Part \ref{part dga computations}, we will compare $\mathcal{C}_{p,q_1}$ and $\mathcal{C}_{p,q_2}$. 
For convenience, we use subscripts, for example, the generators of $\mathcal{C}_{p,q_i}$ are $x_i, y_i, z_i$. 

In order to compare $\mathcal{C}_{p,q_1}$ and $\mathcal{C}_{p,q_2}$, we will define various differential graded algebras, $\mathbb{Z}$ modules, and maps on them. 
We will give definitions of them when we use them, but we would like to introduce a general criterion for them. 
The criterion is that if the generators of the new dgas are naturally related to the generators $x,y,z$ of $\mathcal{C}_{p,q}$, then let $\alpha, \beta, \gamma$ denote the generators corresponding to $x,y,z$ respectively.

\section{Proof of Theorem \ref{thm homotopy type to quasi}}
\label{section theorem 1}
We prove that if $L(p,q_1)$ and $L(p,q_2)$ are of the same homotopy type, then $\mathcal{C}_{p,q_1}$ and $\mathcal{C}_{p,q_2}$ are quasi-equivalent. 

\subsection{Properties of $\chi$}
\label{subsect properties of chi}
By Proposition \ref{prp:cpq-cohomology}, we have that
\begin{gather}
	\label{eqn H^K_p,q}
	H^k_{p,q} = \begin{cases}
		\mathbb{Z}[\mathbb{Z}_p],& \text{  if  } k = -2n \text{  for some  } n \in \Z_{\geq 0},\\ 
		0, &\text{  otherwise}.	
	\end{cases}
\end{gather}
In Sections \ref{subsect properties of chi}--\ref{subsect generators}, we prove that 
\[\{[\chi^n], [x \chi^n], \cdots, [x^{p-1}\chi^n]\}\]
generates $H^{-2n}_{p,q}$.

First, we construct a differential graded algebra $\mathcal{E}$ generated by two elements $\alpha, \gamma$ satisfying 
\begin{itemize}
	\item $|\alpha| = 0, |\gamma| = -2$,
	\item the differential $\partial$ is the zero map, 
	\item $\alpha \gamma = \gamma \alpha$, and
	\item $\alpha^p = 1$. 
\end{itemize}
Moreover, we set a map $\pi : \mathcal{C}_{p,q} \to \mathcal{E}$ as follows.
\begin{gather}
	\label{eqn pi}
	\pi(x) = \alpha, \hspace{1em} \pi(y) = 0, \hspace{1em} \pi(z) = \gamma.
\end{gather}
We note that $\mathcal{C}_{p,q}$ is generated by $x, y, z$ as an algebra without any relations, thus, Equation \eqref{eqn pi} is enough to define a strictly unital algebra map $\pi$. 

\begin{lem}
	\label{lem pi is a dga map}
	The above $\pi$ is a dga map, i.e., $\pi \circ d = \partial \circ \pi$. 
\end{lem}
\begin{proof}
	Since $\partial = 0$, it is enough to prove that $\pi \circ d = 0$. 
	Thus, the following computations complete the proof. 
	\begin{gather*}
		\pi(dx) = \pi(0) = 0, \\
		\pi(dy) = \pi(1-x^p) = 1 - \alpha^p = 0, \\
		\pi(dz) = \pi(x^q y - y x^q) = \alpha^q \cdot 0 - 0 \cdot \alpha^q = 0. 
	\end{gather*}
\end{proof}

The dga map $\pi$ induces a map $H^*\pi$ from $H^*_{p,q}$ to $H^*\mathcal{E} = \mathcal{E}^*$. 
The last equality comes from the fact that $\partial =0$.
With the induced map $H^*\pi$, one can prove Lemma \ref{lem chi is not zero}.

\begin{lem}
	\label{lem chi is not zero}
	In $H^{-2n}_{p,q}$, $[x^k\chi^n]$ is not zero for all $n, k \in \Z_{\geq 0}$.
\end{lem}
\begin{proof}
	We note that from Equation \eqref{eqn d chi}, $[\chi] \in H^{-2}_{p,q}$. 
	Thus, $[\chi^n] \in H^{-2n}_{p,q}$. 
	Then, one can apply $H^*\pi$ for $[\chi^n]$, and one obtains
	\begin{align*}
		H^*\pi([\chi^n]) = [\pi(\chi^n)] = \pi(\chi^n) = \big(\sum_{i=1}^p \alpha^{q(p-1)} \gamma\big)^n = p^n \alpha^{-nq} \gamma^n \neq 0.
	\end{align*}
	We note that $\alpha^p =1$, thus, the negative power of $\alpha$ makes sense. 
	Thus, $[\chi^n] \neq 0$. 
	Similarly, $[x^k \chi^n] \neq 0$ in $H^{-2n}_{p,q}$. 
\end{proof}

Since $[x^k \chi^n] \neq 0$, the statement of Lemma \ref{lem linearly independent} makes sense.
\begin{lem}
	\label{lem linearly independent}
	In $H^{-2n}_{p,q}$, 
	\[\{[\chi^n], [x \chi^n], \cdots, [x^{p-1}\chi^n]\}\]
	is a linearly independent set. 
\end{lem}
\begin{proof}
	We apply $H^*\pi$ for the set, then the result, after reordering, is 
	\[\{p^n\gamma^n, p^n \alpha \gamma^n, \cdots, p^n \alpha^{p-1}\gamma^n\}.\]
	Since the above set is a linearly independent set in $\mathcal{E}$, Lemma \ref{lem linearly independent} is true.
\end{proof}

\begin{rmk}
	\label{rmk base ring}
	Let $\tilde{C}_{p,q}$ (resp.\ $\tilde{\mathcal{E}}$) be the differential graded algebra such that 
	\begin{itemize}
		\item the generators, relations, and the differential map are the same as $\mathcal{C}_{p,q}$ (resp.\ $\mathcal{E}$), but
		\item the base ring is a field whose characteristic is not equal to $p$. 
	\end{itemize}
	On $\tilde{\mathcal{C}}_{p,q}$, one could prove the modified version of Lemma \ref{lem linearly independent} by the same way.
	Then, this is enough to prove that
	\[\{[\chi^n], [x \chi^n], \cdots, [x^{p-1}\chi^n]\}\]
	is a basis of $H^{-2n}_{p,q}$.
\end{rmk}

\subsection{Lemma \ref{lem commuting}}
\label{subsect commuting lemma}
We would like to prove that 
\[\{[\chi^n], [x \chi^n], \cdots, [x^{p-1}\chi^n]\}\]
is a basis for $H^{-2n}_{p,q}$. 
This fact will play a key role in proving Theorem \ref{thm homotopy type to quasi}.
In order to prove that, we apply $\pi$ for the set. 
Then, from Section \ref{subsect properties of chi}, we observe that after applying $\pi$, every member of the set is a multiple of $p^n$. 
Thus, the following proposition is needed.

\begin{prp}
	\label{prop divided by p^n}
	Let $u \in \mathcal{C}^{-2n}_{p,q}$ such that $d u =0$.
	Then, $\pi(u)$ is divisible by $p^n$. 
\end{prp}

Proposition \ref{prop divided by p^n} will be proven in Section \ref{subsect proposition}.
In the current subsection, we prepare the proof of Proposition \ref{prop divided by p^n}.

In order to prove Proposition \ref{prop divided by p^n}, we construct $\mathbb{Z}$-module maps defined on $\mathcal{C}^{-2n}_{p,q}$ and $\mathcal{C}^{-2n+1}_{p,q}$. 
It would be easy to define the maps if we fix $\mathbb{Z}$-module bases for $\mathcal{C}^{-2n}_{p,q}$ and $\mathcal{C}^{-2n+1}_{p,q}$. 
Thus, we fix the following basis $B_N \subset \mathcal{C}^{-N}_{p,q}$, 
\begin{gather}
	\label{eqn basis}
	B_N := \{x^{i_0}Y_1x^{i_1}Y_2x^{i_2}\cdots Y_kx^{i_k} \hspace{0.2em} | \hspace{0.2em} \text{  for all  } j, i_j \geq 0, Y_j = y \text{  or  } z,\\
	\notag \text{  degree of each element of  } B_N \text{  is  } N\}.
\end{gather}

From here to the end of Section \ref{subsect commuting lemma}, we fix a positive integer $n$. 
\begin{dfn}
	\label{dfn conditions}
	\mbox{}
	\begin{enumerate}
		\item Let $b \in B_{2n}$, and let $i$ be an integer such that $0 \leq i \leq n$. 
		Then, $b$ {\em satisfies the condition $\mathfrak{A}_i$} if
		\begin{gather*}
			b = x^{i_0}Y_1x^{i_1}Y_2x^{i_2}\cdots Y_{n+i}x^{i_{n+i}}, \text{  so that  } \\
			Y_1 = \cdots = Y_{2i} = y, \hspace{1em} Y_{2i+1} = \cdots = Y_{n+i} = z.
		\end{gather*}
		\item Let $b = x^{i_0}Y_1x^{i_1}Y_2x^{i_2}\cdots Y_kx^{i_k} \in B_{2n}$, and let $m, i$ be integers such that $0 \leq m \leq p-1$ and $0 \leq i \leq n$. 
		Then, $b$ {\em satisfies the condition $\mathfrak{B}_{m,i}$} if
		\[\sum_{j=0}^k i_j \equiv m+iq \text{  (mod  } p).\]  
		\item Let $b \in B_{2n-1}$, and let $i$ be an integer such that $1 \leq i \leq n$.
		Then, $b$ {\em satisfies the condition $\mathfrak{C}_i$} if 
		\begin{gather*}
			b = x^{i_0}Y_1x^{i_1}Y_2x^{i_2}\cdots Y_{n+i}x^{i_{n+i-1}}, \text{  so that  } \\
			Y_1 = \cdots = Y_{2i-1} = y, \hspace{1em} Y_{2i} = \cdots = Y_{n+i-1} = z.
		\end{gather*}
		\item Let $b = x^{i_0}Y_1x^{i_1}Y_2x^{i_2}\cdots Y_kx^{i_k} \in B_{2n-1}$, and let $m, i$ be integers such that $0 \leq m \leq p-1$ and $1 \leq i \leq n$. 
		Then, $b$ {\em satisfies the condition $\mathfrak{D}_{m,i}$} if 
		\[\sum_{j=0}^k i_j \equiv m+iq \text{  (mod  } p).\]
	\end{enumerate}
\end{dfn}
We note that conditions $\mathfrak{A}_i, \mathfrak{B}_{m,i}, \mathfrak{C}_i, \mathfrak{D}_{m,i}$ depend on the choice of an integer $n$.
Thus, it would be correct to put $n$ in the notation, but we omit $n$ for the simplicity.

For each integer $m$ such that $0 \leq m \leq p-1$, Definition \ref{def Psi and Phi} defines $\mathbb{Z}$-module maps defined on $\mathcal{C}^{2n}_{p,q}$ and $\mathcal{C}^{2n-1}_{p,q}$. 
\begin{dfn}
	\label{def Psi and Phi}
	\mbox{}
	\begin{enumerate}
		\item {\em A $\mathbb{Z}$-module morphism} $\Psi_m : \mathcal{C}^{-2n}_{p,q} \to \mathbb{Z}^{n \times 2}$ is defined as follows:
		let $\psi^{i,j}_m$ be the $(i,j)$ component of $\Psi_m$, and let $b \in B_{2n}$, then 
		\begin{gather*}
			\psi^{i,1}_m(b) := \begin{cases}
				1,&	\text{  if  } b \text{  satisfies conditions  } \mathfrak{A}_{i-1} \text{  and  } \mathfrak{B}_{m, i-1},\\ 
				0,&	\text{  otherwise,} 
			\end{cases}\\
			\psi^{i,2}_m(b) := \begin{cases}
				1,&	\text{  if  } b \text{  satisfies conditions  } \mathfrak{A}_{i} \text{  and  } \mathfrak{B}_{m, i},\\ 
				0,&	\text{  otherwise.}
			\end{cases}
		\end{gather*}
		\item {\em A $\mathbb{Z}$-module morphism} $\Phi_m : \mathcal{C}^{-2n+1}_{p,q} \to \mathbb{Z}^{n}$ is defined as follows:
		let $\phi^i_m$ be the $i^{\text{th}}$ component of $\Phi_m$, and let $b \in B_{2n-1}$, then
		\begin{gather*}
			\phi^i_m(b) =\begin{cases}
				i_{2i-1},&	\text{  if  } b \text{  satisfies conditions  } \mathfrak{C}_{i} \text{  and  } \mathfrak{D}_{m, i},\\ 
				0,&	\text{  otherwise.}
			\end{cases}
		\end{gather*}
	\end{enumerate}
\end{dfn}
\begin{rmk}
	\label{rmk Phi}
	In Definition \ref{def Psi and Phi}.\ 2, the number $i_{2n-1}$ is well-defined since $b \in B_{2n-1}$ where $B_{2n-1}$ is defined in Equation \eqref{eqn basis}.
	We would like to point out that $i_{2n-1}$ is the power of $x$ which appears in the right of the last $y$ and in the left of the first $z$ in $b$ satisfying $\mathfrak{C}_{i}$ and $\mathfrak{D}_{m,i}$.  
\end{rmk}

The $\mathbb{Z}$-module maps $\Psi_m$ and $\Phi_m$ satisfy Lemma \ref{lem commuting}.

\begin{lem}
	\label{lem commuting}
	For any integer $0 \leq m \leq p-1$, $\rho \circ \Psi_m = \Phi_m \circ d$, where 
	\[\rho : \mathbb{Z}^{n \times 2} \to \mathbb{Z}^n, \hspace{1em} A \mapsto A \bigl(\begin{smallmatrix}
		-q	\\ 
		p	
	\end{smallmatrix}\bigr).\]
\end{lem}
\begin{proof}
	We prove this by computing $(\rho \circ \Psi_m)(b)$ and $(\Phi_m \circ d)(b)$ for all $b \in B_{2n}$. 
	\vskip0.2in
	
	\noindent{\em The left hand side, $\rho \circ \Psi_m$.}
	From Definition \ref{def Psi and Phi}, one could observe that there is no $i$ such that $b \in B_{2n}$ satisfies $\mathfrak{A}_i$ and $\mathfrak{B}_{m,i}$, $\Psi_m(b) = 0$. 
	
	Let assume that $b$ satisfies $\mathfrak{A}_i$ and $\mathfrak{B}_{m,i}$ for some $i$ such that $0 \leq i \leq n$.
	Then, $\Psi_m(b) = (\psi_m^{j,k})_{1 \leq j \leq n}^{1 \leq k \leq 2}$ is a matrix such that 
	\begin{itemize}
		\item if $i = 0$, 
		\[\psi_m^{j,k}(b) = \begin{cases}	
			1,& \text{  if  } (j,k) = (1,1),\\ 
			0,& \text{  otherwise,}	
		
		\end{cases}\]
		\item if $1 \leq i < n$, 
		\[\psi_m^{j,k}(b) = \begin{cases}
			1,& \text{  if  } (j,k) = (i,2) \text{  or  } (i+1,1),\\ 
			0,& \text{  otherwise,}	
		\end{cases}\]
		\item if $i = n$, 
		\[\psi_m^{j,k}(b) = \begin{cases}
			1,& \text{  if  } (j,k) = (n,2),\\ 
			0,& \text{  otherwise.}	
		\end{cases}\] 
	\end{itemize}
	
	The above arguments conclude that  
	\begin{align*}
		(\rho \circ \Psi_m) (b) = \begin{cases}
		(-q, 0, \cdots, 0)^T,& \text{  if  $b$ satisfies } \mathfrak{A}_0 \text{  and  } \mathfrak{B}_{m,0},\\ 
			(0, \cdots, 0, p, -q, 0, \cdots, 0)^T,& \text{  if $b$ satisfies  } \mathfrak{A}_i \text{  and  } \mathfrak{B}_{m,i} \text{  for  } 1 \leq i <n, \\
			(0, \cdots, 0, p)^T,& \text{  if  $b$ satisfies } \mathfrak{A}_n \text{  and  } \mathfrak{B}_{m,n},\\
			(0, \cdots, 0)^T,& \text{  otherwise,}
		\end{cases}
	\end{align*} 
	where $T$ stands for transposing a row matrix to a column matrix.
	\vskip0.2in
	
	\noindent{\em The right hand side, $\Phi \circ d$.}
	We note that $d (b)$ could be written as a linear combination of $B_{2n-1}$ uniquely.
	If $\phi_m^i\big(d (b)\big)$ is not zero for some $i$, then it means that the linear combination for $d (b)$ contains $b' \in B_{2n-1}$ such that $b'$ satisfies the conditions $\mathfrak{C}_i$ and $\mathfrak{D}_{m,i}$.	
	
	If $d (b)$ of $b \in B_{2n}$ contains $b' \in B_{2n-1}$ satisfying $\mathfrak{C}_i$, then one of the followings holds:
	\begin{enumerate}
		\item[(i)] $b$ satisfies $\mathfrak{A}_i$, and $b'$ is obtained by taking the derivative of one of $y$, 
		\item[(ii)] $b$ satisfies $\mathfrak{A}_{i-1}$, and $b'$ is obtained by taking the derivative of the first $z$,
		\item[(iii)] $b = x^{i_0}Y_1x^{i_1}\cdots x^{i_{n+i-2}}Y_{n+i-1} x^{i_{n+i-1}}$ such that 
		\[Y_1 = \cdots =Y_{k-1} = y = Y_{k+1} = \cdots = Y_{2i-1}, Y_k = z = Y_{2i} = \cdots = Y_{n+i-1},\]
		for some $k$ such that $1 \leq k < 2i-1$, and $b'$ is obtained by taking the derivative of $Y_k = z$, or
		\item[(iv)] $b = x^{i_0}Y_1x^{i_1}\cdots x^{i_{n+i-1}}Y_{n+i} x^{i_{n+i}}$ such that 
		\[Y_1 = \cdots = Y_{2i-1} = y = Y_{2i+k}, Y_{2i} = \cdots = Y_{2i+k-1} = z = Y_{2i+k+1} = \cdots = Y_{n+i}, \]
		for some $k$ such that $1 \leq k \leq n-i$, and $b'$ is obtained by taking the derivative of $Y_{2i+k} = y$.
	\end{enumerate}
	\vskip0.2in
	
	{\em If (i) holds}, i.e., 
	\[b = x^{i_0}yx^{i_1}\cdots y x^{i_{2i}} z \cdots z x^{i_{n+i}},\]
	then,
	\begin{align*}
		\phi_m^i\big(d (b)\big) &= \phi_m^i\big( \sum_{k=1}^{2i} (-1)^{k-1} 
		x^{i_0}yx^{i_1}\cdots x^{i_{k-1}}(1-x^p)x^{i_{k}} \cdots y x^{i_{2i}} z \cdots z x^{i_{n+i}}\big) \\
		&= \sum_{k=1}^{2i} (-1)^{k-1}\big( \phi_m^i( x^{i_0}yx^{i_1}\cdots x^{i_{k-1} + i_{k}} \cdots y x^{i_{2i}} z \cdots z x^{i_{n+i}}) \\ & \hspace{7em} - \phi_m^i(x^{i_0}yx^{i_1}\cdots x^{i_{k-1} + i_{k} + p} \cdots y x^{i_{2i}} z \cdots z x^{i_{n+i}}) \big).
	\end{align*}
	If the above one is not zero, then the terms in $\phi^i_m$ should satisfy $\mathfrak{D}_{m,i}$. 
	In other words, 
	\[\sum_{k=0}^{n+i} i_k \equiv m+iq \text{  (mod  } p),\]
	or equivalently, $b$ satisfies $\mathfrak{B}_{m,i}$. 
	If $b$ satisfies $\mathfrak{B}_{m,i}$, then 
	\begin{align*}
		\phi_m^i\big(d (b)\big) = \sum_{k=1}^{2i-1} (-1)^k(i_{2i} - i_{2i}) + (-1)^{2i-1}(i_{2i-1} + i_{2i} - i_{2i-1} - i_{2i} -p) = p.
	\end{align*}
	Shortly, if $b$ satisfies $\mathfrak{A}_i$,
	\[\phi_m^i\big(d (b)\big) = \begin{cases}
		p,& \text{  if  $b$ satisfies } \mathfrak{B}_{m,i},\\ 
		0,& \text{  otherwise.}	
	\end{cases}\]
	\vskip0.2in
	
	{\em If (ii) holds}, i.e.,  
	\[b = x^{i_0}yx^{i_1}\cdots y x^{i_{2i-2}} z x^{i_{2i-1}}\cdots z x^{i_{n+i-1}},\] 
	then, similarly to the case of (i), one obtains
	\begin{align*}
		\phi_m^i\big(d (b)\big) &= \phi_m^i\big((-1)^{2i-2} x^{i_0}yx^{i_1}\cdots y x^{i_{2i-2}} (x^qy - yx^q) x^{i_{2i-1}} z\cdots z x^{i_{n+i-1}}\big) \\
		&= \begin{cases}
			-q,& \text{  if  $b$ satisfies } \mathfrak{B}_{m,i-1},\\ 
			0,& \text{  otherwise.}	
		\end{cases}
	\end{align*}
	\vskip0.2in
	
	{\em If (iii) holds}, then 
	\[b= x^{i_0}yx^{i_1}\cdots y x^{i_{k-1}} z x^{i_k} y \cdots y x^{i_{2i-1}} z x^{i_{2i}}\cdots z x^{i_{n+i}}, \text{  for some  } k \text{  such that  } 1 \leq k <2i-1. \]
	Similar to the above computations, 
	\begin{align*}
		\phi_m^i\big(d (b)\big) &= \phi_m^i((-1)^{k-1} x^{i_0}yx^{i_1}\cdots y x^{i_{k-1}} (x^qy-yx^q) x^{i_k} y \cdots y x^{i_{2i-1}} z x^{2i}\cdots z x^{i_{n+i-1}} ) \\
		&= \begin{cases}
			(-1)^{k-1} (i_{2i-1} - i_{2i-1}) = 0,& \text{  if  $b$ satisfies }  \mathfrak{B}_{m,i-1},\\ 
			0,& \text{  otherwise.}
		\end{cases}
	\end{align*}
	\vskip0.2in
	
	{\em If (iv) holds}, then
	\[b = x^{i_0}yx^{i_1}\cdots y x^{i_{2i-1}} z x^{i_{2i}}\cdots x^{i_{2i+k-1}} y x^{i_{2i+k}} z \cdots z x^{i_{n+i}}, \text{  for some  } k \text{  such that  } 1 \leq k \leq n-i.\]
	Similar to the cases of (i)--(iii), one obtains
	\begin{align*}
		\phi_m^i(d b) & = \phi_m^i\big((-1)^{2i-1}x^{i_0}yx^{i_1}\cdots y x^{i_{2i-1}} z x^{2i}\cdots x^{2i+k-1} (1- x^p) x^{2i+k} z \cdots z x^{i_{n+i}}\big) \\
		&=\begin{cases}
			(-1)^{2i-1} (i_{2i-1} - i_{2i-1}) = 0,& \text{  if  $b$ satisfies }  \mathfrak{B}_{m,i},\\ 
			0,& \text{  otherwise.}
		\end{cases}
	\end{align*}
	
	Simply, one concludes that 
	\begin{align*}
		(\Phi_m \circ d) (b) =\begin{cases}	(-q, 0, \cdots, 0)^T,& \text{  if  $b$ satisfies } \mathfrak{A}_0 \text{  and  } \mathfrak{B}_{m,0},\\ 
			(0, \cdots, 0, p, -q, 0, \cdots, 0)^T,& \text{  if $b$ satisfies  } \mathfrak{A}_i \text{  and  } \mathfrak{B}_{m,i} \text{  for some  } i \text{  such that  } 1 \leq i <n, \\
			(0, \cdots, 0, p)^T,& \text{  if  $b$ satisfies } \mathfrak{A}_n \text{  and  } \mathfrak{B}_{m,n},\\
			(0, \cdots, 0)^T,& \text{  otherwise.}
		\end{cases}
	\end{align*} 
	This completes the proof.
\end{proof}

\subsection{Proposition \ref{prop divided by p^n}}
\label{subsect proposition}
We prove Proposition \ref{prop divided by p^n} in the current subsection.

\begin{proof}[Proof of Proposition \ref{prop divided by p^n}.]
	We recall that 
	\[\pi(x) = \alpha, \pi(y) = 0, \pi(z)= \gamma.\]
	Thus, for any $b \in B_{2n}$ defined in Equation \eqref{eqn basis}, $\pi(b) = 0$ if $b$ contains at least one $y$. 
	In other words, $\pi(b) \neq 0$ if and only if $b$ satisfies $\mathfrak{A}_0$.
	Moreover, if $b \in B_{2n}$ satisfies $\mathfrak{A}_0$ and $\mathfrak{B}_{m,0}$, then, $\pi(b) = \alpha^m \gamma^n$.
	
	We also note that $\pi(u) \in \mathcal{E}^{-2n}$, and that $\mathcal{E}^{-2n}$ is generated by 
	\[\{\gamma, \alpha \gamma, \cdots, \alpha^{p-1} \gamma\},\]
	as a $\mathbb{Z}$ module.
	Thus, there are $a_i \in \mathbb{Z}$ for $i = 0, \cdots, p-1$ such that 
	\[\pi(u) = \sum_{i=0}^{p-1}a_i \alpha^i \gamma^n.\]
	Under this, $p^n$ divides $\pi(u)$ if and only if $p^n$ divides $a_i$ for all $i$. 
	We will show that $p^n$ divides $a_m$ by using $\Psi_m$ and $\Phi_m$ 
	
	Since $B_{2n}$ is a generating set of $\mathcal{C}_{p,q}^{-2n}$, $u$ can be uniquely written as 
	\[u = \sum_{b \in B_{2n}} c_b b,\]
	for some $c_b\in\Z$. From the above arguments, $a_m$ is the sum of $c_b$ such that $b$ satisfies $\mathfrak{A}_0$ and $\mathfrak{B}_{m,0}$.  
	
	Let 
	\[\Psi_m(u) = \begin{pmatrix}
		b_0	& b_1 \\ 
		b_1	& b_2 \\ 
		\vdots	& \vdots \\ 
		b_{n-1}	& b_n
	\end{pmatrix},\]
	for some $b_i\in\Z$. We note that $\Psi_m(u)$ is in the above form by definition of $\Psi_m$.
	Then, it is easy to observe that $a_m = b_0$. 
	
	Since $d u = 0$, 
	\[(\rho \circ \Psi_m)(u) = (\Phi_m \circ d)(u) = (0, \cdots, 0)^T.\]
	Thus, 
	\[q b_{i-1} = p b_i, \text{  for all  } 1 \leq i \leq n.\]
	It concludes that $q^n b_0 = p^n b_n$. 
	Since $p$ and $q$ are relatively prime to each other, $p^n$ divides $b_0 = a_m$. 
\end{proof}

\subsection{Generators}
\label{subsect generators}
We prove Proposition \ref{prop generators} in the current subsection.

\begin{prp}
	\label{prop generators}
	The set 
	\[I := \{[\chi^n], [x \chi^n], \cdots, [x^{p-1}\chi^n]\}\]
	generates $H^{-2n}_{p,q}$. 
\end{prp}
\begin{proof}
	Let $\tilde{C}_{p,q}$ (resp.\ $\tilde{\mathcal{E}}$) be the differential graded algebra such that 
	\begin{itemize}
		\item the generators, relations, and the differential map are the same as $\mathcal{C}_{p,q}$ (resp.\ $\mathcal{E}$), but
		\item the base ring is the field of rational numbers $\mathbb{Q}$. 
	\end{itemize}
	Then, there is a natural embedding of $\mathcal{C}_{p,q}$ (resp.\ $\mathcal{E}$) to $\tilde{\mathcal{C}}_{p,q}$ (resp.\ $\tilde{\mathcal{E}}$).
	Thus, $I$ is a subset of $\tilde{\mathcal{C}}_{p,q}$
	Moreover, as mentioned in Remark \ref{rmk base ring}, the set $I$ generates $H^{-2n}(\tilde{\mathcal{C}}_{p,q})$. 
	
	This implies that, if $[u] \in H^{-2n}_{p,q}$, then, $u$ can be written as a linear combination of $I$ with rational coefficients since $[u] \in H^{-2n}(\tilde{\mathcal{C}}_{p,q})$. 
	It means that there is $c \in \mathbb{Z}$ such that $c [u]$ can be written as a linear combination of $I$ with integer coefficients. 
	Let $c_0$ be the smallest positive integer among such $c$. 
	Then, in order to prove Proposition \ref{prop generators}, it is enough to show that $c_0 =1$.
	
	Let $c_0 [u] = \sum_{i=0}^{p-1} a_i [x^i \chi^n]$ with $a_i \in \mathbb{Z}$. 
	Then, 
	\begin{align*}
		c_0 \pi(u) &= \pi(c_0u) \\
		&= H^*\pi(c_0 [u]) \\
		&= H^*\pi( \sum_{i=0}^{p-1} a_i [x^i \chi^n]) \\
		&= \sum_{i=0}^{p-1} \big(a_i p^n \alpha^{i-nq} \gamma^n\big). 
	\end{align*}	
	We note that $\alpha^p=1$ in $\mathcal{E}$, thus, the negative power of $\alpha$ makes sense. 
	
	On the other hand, let 
	\[\pi(u) = \sum_{i=0}^{p-1} b_i \alpha^i \gamma^n,\]
	for some $b_i \in \mathbb{Z}$. 
	This is possible because $\pi(u) \in \mathcal{E}^{-2n}$ which is generated by 
	\[\{\gamma^n, \alpha \gamma^n, \cdots, \alpha^{p-1} \gamma^n\}.\]
	Then, after relabeling $a_i$, one obtains 
	\[\sum_{i=0}^{p-1} c_0 b_i \alpha^i \gamma^n = \sum_{i=0}^{p-1} a_i p^n \alpha^i \gamma^n. \]
	Thus, $c_0 b_i = p^n a_i$ for all $i$. 
	
	By Proposition \ref{prop divided by p^n}, $\displaystyle \frac{b_i}{p^n}$ is an integer.
	Thus, from $\displaystyle c_0 \frac{b_i}{p^n} = a_i$, one concludes that $c_0$ divides $a_i$ for all $i$. 
	
	Since we choose the smallest positive $c_0$, $c_0$ cannot divide $a_i$ for all $i$ unless $c_0 =1$. 
	Thus, $c_0 =1$, and it completes the proof.   
\end{proof}

\begin{cor}
	\label{cor 1}
	If the base ring is a field whose characteristic is not equal to $p$ (resp.\ $\mathbb{Z}$), then, $\pi$ is quasi-equivalence (resp.\ quasi-faithful).
\end{cor}

\begin{proof}
	Under $H^{-2n}\pi\colon H_{p,q}^{-2n}\to\cE^{-2n}$, the basis
	\[\{[\chi^n], [x \chi^n], \cdots, [x^{p-1}\chi^n]\}\]
	of $H_{p,q}$ is mapped to the set
	\[\{p^n\gamma^n, p^n \alpha \gamma^n, \cdots, p^n \alpha^{p-1}\gamma^n\}\]
	in $\mathcal{E}$. Since this set is linearly independent in $\cE$ for the base ring $\Z$, $\pi$ is quasi-faithful. If the base ring is field of characteristic not equal to $p$, then the latter set is a basis for $\cE^{-2n}$ since $p^n$ is a unit, hence $\pi$ is a quasi-equivalence.
\end{proof}

\begin{rmk}\label{rmk:lens-field}
	Thus, when the base ring is a field whose characteristic is not $p$, $\cC_{p,q_1}$ and $\cC_{p,q_2}$ are quasi-equivalent for any $q_1,q_2$. Hence Theorem \ref{thm homotopy type to quasi} is trivial and Theorem \ref{thm quasi to homotopy type} is not true. 
	This is the reason why we do not choose a field coefficient. When the base ring is $\Z$, we will use quasi-faithfulness of $\pi$ to investigate our dga $\cC_{p,q}$.
\end{rmk}

\begin{rmk}\label{rmk:product-lens}
	One can also easily observe that, since
	\[\{[\chi^n], [x \chi^n], \cdots, [x^{p-1}\chi^n]\}\]
	is a basis for $H^{-2n}_{p,q}$, the product structure on the cohomology $H_{p,q}$ of $\cC_{p,q}$ does not depend on $q$, hence it does not distinguish non-homotopic lens spaces.
\end{rmk}

\subsection{Theorem \ref{thm homotopy type to quasi}}
\label{subsect Theorem 1}
We prove Theorem \ref{thm homotopy type to quasi} in Section \ref{subsect Theorem 1}.

\begin{proof}[Proof of Theorem \ref{thm homotopy type to quasi}.]
	
	Before starting the proof, we review notation from Section \ref{section setting} and preliminaries from lower dimensional topology. 
	\begin{itemize}
		\item First, in the current subsection, we compare $\mathcal{C}_{p,q_1}$ and $\mathcal{C}_{p,q_2}$. 
		In order to distinguish those two differential graded algebras and their elements, maps, etc, we use subscripts, for example, we use $x_i, y_i, z_i, d_i, \pi_i$ for $i=1,2$. 
		\item Since $p, q_i$ are relatively primes, there are $\bar{q}_i, r_i$ such that $q_i \bar{q}_i = r_i p +1$. See Equation \eqref{eqn p,q}. 
		\item Since $L(p,q_1)$ and $L(p,q_2)$ are of the same homotopy type, by Theorem \ref{thm:lens-classification}, there are integers $a, b, c$ such that $bq_2 = a^2 q_1 + cp$, and such that $b$ is either $1$ or $-1$. 
	\end{itemize}
	
	We prove Theorem \ref{thm homotopy type to quasi} by constructing a specific quasi-equivalence 
	\[F : \mathcal{C}_{p,q_1} \to \mathcal{C}_{p,q_2}.\] 
	
	Since $x_1, y_1, z_1$ generate $\mathcal{C}_{p,q_1}$ as an algebra, an algebra morphism $F$ is defined by setting $F(x_1), F(y_1)$, and $F(z_1)$. 
	Let 
	\begin{align*}
		F(x_1) &:= x_2^a,\\
		F(y_1) &:= y_2 (\sum_{i=1}^{a} x_2^{p(i-1)}), \\
		F(z_1) &:= (\sum_{i=1}^{a q_1} x_2^{aq_1-i} \Lambda_2 x_2^{i-1}) (\sum_{i=1}^{a} x_2^{p(i-1)}) +(c \bar{q}_2 -b r_2)x_2^{a q_1} \chi_2.
	\end{align*}
	We note that with Equation \eqref{eqn f_n(x)}, one can replace $\sum_{i=1}^{a} x_2^{p(i-1)}$ with $f_a(x_2^p)$.
	We also note that $\chi_i$ and $\Lambda_i$ are defined in Equations \eqref{eqn chi} and \eqref{eqn Lambda}, see also Equation \eqref{eqn def chi Lambda}.
	
	In order to prove that $F$ is a dga morphism, it is necessarily to check that $F \circ d_1 = d_2 \circ F$. 
	Thus, we compute the followings: 
	\begin{align*}
		(F \circ d_1) (x_1) &= F( 0) = 0, \\
		(d_2 \circ F)(x_1) &= d_2 (x_2^a) = 0, \\
		(F \circ d_1) (y_1) &= F(1 - x_1^p) = 1 - x_2^{ap}, \\
		(d_2 \circ F)(y_1) &= d_2( y_2 f_a(x_2^p)) = (1-x_2^p) f_a(x_2^p) = 1- x_2^{ap}, \\
		(F \circ d_1) (z_1) &= F( x_1^{q_1} y_1 - y_1 x_1^{q_1}) = x_2^{a q_1}y_2 f_a(x_2^p) - y_2 f_a(x_2^p) x_2^{a q_1}  = (x_2^{a q_1}y_2 - y_2 x_2^{a q_1})f_a(x_2^p), \\
		(d_2\circ F)(z_1) &= d_2\big((\sum_{i=1}^{a q_1} x_2^{aq_1-i} \Lambda_2 x_2^{i-1}) f_a(x_2^p) +(c \bar{q}_2 -b r_2)x_2^{a q_1} \chi_2\big) \\
		&= (\sum_{i=1}^{a q_1} x_2^{aq_1-i} (x_2y_2 - y_2 x_2) x_2^{i-1}) f_a(x_2^p) \\
		&= (x_2^{a q_1} y_2 - y_2 x_2^{a q_1}) f_a(x_2^p).
	\end{align*}
	See Equations \eqref{eqn d chi} and \eqref{eqn d Lambda} for computing $(d_2 \circ F)(z_1)$. 
	It concludes that $F$ is a dga morphism. 
	
	In order to prove that $F$ is a quasi-equivalent map, it is enough to show that for all $n \in \Z_{\geq 0}$, 
	\[H^{-2n}F : H^{-2n}_{p,q_1} \to H^{-2n}_{p,q_2}\]
	is an isomorphism, where $H^*F$ is the induced morphism. 
	In other words, it is enough to concentrate on the even degrees.
	This is because of Equation \eqref{eqn H^K_p,q}. 
	
	Since $b$ is either $1$ or $-1$, if 
	\begin{gather}
		\label{eqn goal}
		H^*F(\{[\chi_1^n], [x_1 \chi_1^n], \cdots, [x_1^{p-1} \chi_1^n]\}) =	\{b^n[\chi_2^n], b^n[x_2 \chi_2^n], \cdots, b^n[x_2^{p-1} \chi_2^n]\},
	\end{gather}
	then, it completes the proof by Proposition \ref{prop generators}. 
	Thus, we will show that Equation \eqref{eqn goal} holds.
	
	First, we compute the following: 
	\begin{align}
		\label{eqn Hpi_2}
		H^*\pi_2 ([F(\chi_1)]) &= \pi_2\big(F(\chi_1)\big) \\
		\notag	&= \pi_2\big(F(y_1 f_{q_1}(x_1^p)y_1 + \sum_{i=1}^p x_1^{q_1(p-i)} z_1 x_1^{q_1(i-1)})\big) \\
		\notag	&= \pi_2\big(y_2 f_a(x_2^p) f_{q_1}(x_2^{ap})y_2 f_a(x_2^p)\big) + \sum_{i=1}^p \pi_2 \big(F(x_1^{q_1(p-i)} z_1 x_1^{q_1(i-1)} )\big) \\
		\notag	&= \sum_{i=1}^p \alpha^{aq_1(p-1)} \pi_2\big(F(z_1)\big) \\
		\notag	&= p \alpha^{aq_1(p-1)} \big[\pi_2\big((\sum_{i=1}^{a q_1} x_2^{aq_1-i} \Lambda_2 x_2^{i-1}) f_a(x_2^p) +(c \bar{q}_2 -b r_2)x_2^{a q_1} \chi_2\big)\big] \\
		\notag	&=p\alpha^{aq_1(p-1)} \big[\big(\sum_{i=1}^{aq_1}\alpha^{aq_1-1} \pi_2(\Lambda_2)\big) f_a(\alpha^p) + (c \bar{q}_2 - b r_2) \alpha^{aq_1} \pi_2(\chi_2)\big] \\
		\notag	&= p\alpha^{aq_1(p-1)} \big[a q_1 \alpha^{aq_1-1} (\sum_{i=1}^{\bar{q}_2} \alpha^{q_2(\bar{q}_2-1)}\gamma) f_a(\alpha^p) + (c\bar{q}_2-br_2) \alpha^{aq_1} (\sum_{i=1}^p \alpha^{q_2(p-1)}\gamma) \big] \\
		\notag	&= pa q_1 \bar{q}_2 \alpha^{q_2 \bar{q}_2 - q_2 -1} f_a(\alpha^p) \gamma + p^2(c \bar{q}_2 - b r_2) \alpha^{-q_2} \gamma.
	\end{align}
	We note that the negative power of $\alpha$ makes sense since $\alpha^p=1$. 
	
	One can easily check the followings facts: 
	\begin{itemize}
		\item since $\alpha^p =1$, $f_a(\alpha^p) =f_a(1) =a$,
		\item since $b q_2 = a^2 q_1 + cp$, and since $q_2 \bar{q}_2 = r_2 p +1$,  
		\[a^2 q_1 \bar{q_2} = b q_2 \bar{q}_2 - cp \bar{q}_2 = b (r_2 p +1) -cp \bar{q}_2,\] 
		\item since $q_2 \bar{q}_2 = r_2 p +1$, $\alpha^{q_2 \bar{q}_2 - q_2 -1} = \alpha^{r_2 p +1 -q_2 -1} = \alpha^{-q_2}$. 
	\end{itemize}
	
	The given facts induce that Equation \eqref{eqn Hpi_2} becomes
	\begin{align}
		\label{eqn final} 
		H^*\pi_2([F(\chi_1)]) &= pa q_1 \bar{q}_2 \alpha^{q_2 \bar{q}_2 - q_2 -1} f_a(\alpha^p) \gamma + p^2(c \bar{q}_2 - b r_2) \alpha^{-q_2} \gamma \\
		\notag &= p a^2 q_1 \bar{q}_2 \alpha^{-q_2}\gamma + p^2(c \bar{q}_2 - b r_2)\alpha^{-q_2} \gamma \\
		\notag &= \big(pb(r_2p +1) - p^2 c \bar{q}_2\big)\alpha^{- q_2} \gamma + p^2(c \bar{q}_2-br_2)\alpha^{-q_2}\gamma \\
		\notag &= pb \alpha^{-q_2} \gamma. 
	\end{align}
	
	By Proposition \ref{prop generators}, there are $a_0, \cdots, a_{p-1} \in \mathbb{Z}$ such that 
	\begin{align}
		\label{eqn the other hand}
		[F(\chi_1)] = \sum_{i=0}^{p-1} a_i [x_2^i \chi_2]. 
	\end{align}
	
	From Equations \eqref{eqn final} and \eqref{eqn the other hand}, one concludes that 
	\begin{align}
		\label{eqn injectivity}
		p b\alpha^{-q_2} \gamma = H^* \pi_2([F(\chi_1)]) = \pi_2\big(F(\chi_1)\big) = \sum_{i=0}^{p-1} p a_i \alpha^{i-q_2} \gamma. 
	\end{align}
	A more conceptual way of concluding this is the fact that $\pi$ is quasi-faithful.
	
	Equation \eqref{eqn injectivity} implies that $a_0 =b$, and $a_i =0 $ for all $i \neq 0$.
	In other words, $[F(\chi_1)] = b[\chi_2]$. 
	Thus, 
	\begin{align}
		\label{eqn result}
		H^*F(\{[\chi_1], [x_1 \chi_1], \cdots, [x_1^{p-1} \chi_1]\}) = \{b[\chi_2], b[x_2^a \chi_2], \cdots, b[x_2^{a(p-1)} \chi_2]\}.
	\end{align}
	
	Now, it is enough to show that $a$ and $p$ are relative primes from Equation \eqref{eqn result}. 
	This is because 
	\begin{itemize}
		\item $b q_2 \equiv a^2 q_1 \text{  (mod  } p)$ with $b = \pm 1$,
		\item $q_1, q_2$ are relative primes to $p$.
	\end{itemize} 
	Thus, Equation \eqref{eqn result} induces Equation \eqref{eqn goal} for $n=1$.
	
	For a general $n$, Equation \eqref{eqn goal} holds since 
	\[F([x_1^i \chi_1^n]) = [x_2^{ai} b^n \chi_2^n].\] 
\end{proof}

Theorem \ref{thm homotopy type to quasi} says that the topology of $L(p,q)$, more specifically the homotopy type of $L(p,q)$, determines the wrapped Fukaya category of $T^*L(p,q)$. 
In other words, we obtain Corollary \ref{cor Fukaya category}. 

\begin{proof}[Proof of Corollary \ref{cor Fukaya category}]
	Theorem \ref{thm:wrap-lens} proves that the wrapped Fukaya category of $T^*L(p,q_i)$ is pretriangulated equivalent to $\mathcal{C}_{p,q_i}$. Since quasi-equivalence implies pretriangulated equivalence, Theorem \ref{thm homotopy type to quasi} implies Corollary \ref{cor Fukaya category}. 
\end{proof}

\section{Proof of Theorem \ref{thm quasi to homotopy type}}
\label{section dga}
In Section \ref{section dga}, we prove Theorem \ref{thm quasi to homotopy type}.

\subsection{Setting} 
\label{subsect setting}
We follow the convention and use the same notation, which are given in Section \ref{section setting}. 
We also need another differential graded algebra in Section \ref{section dga}, which is defined in Definition \ref{def D}.

\begin{dfn}
	\label{def D}
	\mbox{}
	\begin{enumerate}
		\item Let {\em $\mathcal{D}$} be the strictly unital graded algebra generated freely by two generators $\beta$ and $\gamma$ such that 
		\[|\beta|=-1, |\gamma|= -2.\]
		\item Let $\partial : \mathcal{D}\to \mathcal{D}$ is given by 
		\begin{gather*}
			\partial(\beta) = \partial(\gamma) = 0.
		\end{gather*}
		The pair $(\mathcal{D}, \partial)$ is a differential graded algebra. 
	\end{enumerate}
\end{dfn}
For convenience, we simply write $\mathcal{D}$ for the differential graded algebra, instead of the pair $(\mathcal{D},\partial)$.

\subsection{DGA Morphisms}
\label{subsect dga morphisms}
Before starting Section \ref{subsect dga morphisms}, we remark that $p$ and $q$ are relatively prime. 
Also, we note that for convenience, we will assume that $p$ is an odd number in Sections \ref{subsect dga morphisms} -- \ref{subsect comparison}.
In Section \ref{subsect even p}, we will discuss the case of even $p$. 

Let $\mu$ be a strictly unital dga morphism from $\mathcal{C}_{p,q}$ to $\mathcal{D}$. 
Then, because of the degree reason, 
\begin{gather}
	\label{eqn dga morphism}
	\mu(x) = a_0, \mu(y) = a \beta, \mu(z) = b \gamma + c \beta^2,
\end{gather}
for some $a_0, a, b, c \in \mathbb{Z}$. 

Since $\mu$ is a dga morphism, $\mu \circ d = \partial \circ \mu$.
Moreover, $\partial =0$, thus
\[(\mu\circ d)(y) = \mu(1-x^p) = 1 - a_0^p = 0.\]
Then, since $p$ is an odd number, $a_0$ should be $1$. 

It is easy to check that if $a_0 =1$, then for any $a, b, c \in \mathbb{Z}$, Equation \eqref{eqn dga morphism} defines a dga map. 
Let $\mu_{a,b,c}$ denote the dga map defined by Equation \eqref{eqn dga morphism} for $a, b, c \in \mathbb{Z}$. 

\begin{lem}
	\label{lemma 1}
	For any $a, b, c \in \mathbb{Z}$, there is a dga morphism $\tilde{\mu}_{a,b,c} : \mathcal{D}\to \mathcal{D}$ such that $\tilde{\mu}_{a,b,c} \circ \mu_{1,1,0} = \mu_{a,b,c}$. 
\end{lem}
\begin{proof}
	Let $\tilde{\mu}$ be a strictly unital graded algebra morphism defined as follows:
	\begin{gather*}
		\tilde{\mu}_{a,b,c}(\beta) = a \beta, \hspace{1em} \tilde{\mu}_{a,b,c}(\gamma) = b \gamma +c \beta^2.
	\end{gather*}
	Since $\mathcal{D}$ is generated by $\beta$ and $\gamma$ freely, the above equations are enough to define a graded algebra morphism $\tilde{\mu}_{a,b,c}$.
	Moreover, $\tilde{\mu}_{a,b,c}$ commutes with the differential map $\partial$ since $\partial$ is the zero map. 
	Thus, $\tilde{\mu}_{a,b,c}$ is a dga morphism.
	Then, it is easy to check that $\tilde{\mu}_{a,b,c} \circ \mu_{1,1,0} = \mu_{a,b,c}$.   
\end{proof}
By Lemma \ref{lemma 1}, we focus on $\mu_{1,1,0}$.
For convenience, let $\mu$ denote $\mu_{1,1,0}$. 

\subsection{$\mathbb{Z}$-module Morphisms}
\label{subsect Z module morphisms}
Let $\mathcal{C}^k_{p,q}, H^k_{p,q}$ and $\mathcal{D}^k$ be the degree $k$ part of $\mathcal{C}_{p,q}$, the cohomology of $\mathcal{C}_{p,q}$, and $\mathcal{D}$, respectively.  
In Section \ref{subsect Z module morphisms}, we concentrate on $\mathcal{C}^{-i}_{p,q}, H^{-i}_{p,q}$ and $\mathcal{D}^{-i}$ for $i =1, 2$.

As a $\mathbb{Z}$-module, $\mathcal{D}^{-2}$ is freely generated by $\beta^2$ and $\gamma$, i.e., 
\[\mathcal{D}^{-2} = \mathbb{Z}\langle \beta^2 \rangle \oplus \mathbb{Z} \langle \gamma \rangle.\]
Thus, it is easy to define a $\mathbb{Z}$-module morphism $g : \mathcal{D}^{-2} \to \mathbb{Z}\oplus \mathbb{Z}$ such that 
\begin{gather}
	\label{eqn g}
	g(\beta^2) = (-p,p), g(\gamma) = (q,-q).
\end{gather} 

Similarly, as a $\mathbb{Z}$-module, $\mathcal{C}^{-1}_{p,q}$ is freely generated by $x^{i_1} y x^{i_2}$ for all $i_1, i_2 \in \Z_{\geq 0}$.
Thus, there is a $\mathbb{Z}$-module morphism $f : \mathcal{C}^{-1}_{p,q} \to \mathbb{Z}\oplus \mathbb{Z}$ such that
\[f(x^{i_1} y x^{i_2}) = (i_1, i_2). \] 
Then, Lemma \ref{lemma 3} holds. 

\begin{lem}
	\label{lemma 3}
	On $\mathcal{C}^{-2}_{p,q}$, $g \circ \mu = f \circ d$. 
\end{lem}
\begin{proof}
	As a $\mathbb{Z}$ module, $\mathcal{C}^{-2}_{p,q}$ is freely generated by $x^{i_1} y x^{i_2} y x^{i_3}$ and $x^{j_1} z x^{j_2}$ for non-negative integers $i_1, i_2, i_3, j_1,j_2$.
	Thus, an arbitrary element $A \in \mathcal{C}^{-2}_{p,q}$ can be  written as a linear combination uniquely, as follow:
	\[A = \sum_{i_1, i_2, i_3} A_y(i_1, i_2, i_3) x^{i_1} y x^{i_2} y x^{i_3} + \sum_{j_1,j_2} A_z(j_1,j_2)x^{j_1} z x^{j_2},\]
	where $A_y(i_1,i_2,i_3)$ and $A_z(j_1,j_2)$ are integers. 
	
	By direct computations, 
	\begin{align*}
		(g \circ \mu)(A) &= g(\sum_{i_1, i_2, i_3} A_y(i_1, i_2, i_3) \beta^2  + \sum_{j_1,j_2} A_z(j_1,j_2)\gamma), \\
		&= \sum_{i_1, i_2, i_3} A_y(i_1, i_2, i_3) (-p,p)  + \sum_{j_1,j_2} A_z(j_1,j_2)(q,-q), \\
		dA &= \sum_{i_1, i_2, i_3} A_y(i_1, i_2, i_3) x^{i_1} (1-x^p) x^{i_2} y x^{i_3} \\
		& \hspace{1em} - \sum_{i_1, i_2, i_3} A_y(i_1, i_2, i_3) x^{i_1} y x^{i_2} (1-x^{p}) x^{i_3} \\
		& \hspace{2em} + \sum_{j_1,j_2} A_z(j_1,j_2)x^{j_1} (x^qy-yx^q) x^{j_2}, \\
		f(dA) &= \sum_{i_1, i_2, i_3} A_y(i_1, i_2, i_3) \big((i_1 + i_2, i_3) - (i_1 + i_2 +p, i_3)\big)\\
		& \hspace{1em}- \sum_{i_1, i_2, i_3} A_y(i_1, i_2, i_3) \big((i_1, i_2 + i_3) - (i_1, i_2 + i_3 +p)\big) \\
		& \hspace{2em} + \sum_{j_1,j_2} A_z(j_1,j_2) \big((j_1 +q, j_2) - (j_1, j_2+q)\big) \\
		&= \sum_{i_1, i_2, i_3} A_y(i_1, i_2, i_3) (-p,p)  + \sum_{j_1,j_2} A_z(j_1,j_2)(q,-q).
	\end{align*}
\end{proof}

Lemma \ref{lemma 5} follows.
\begin{lem}
	\label{lemma 5}
	Let $A \in \mathcal{C}^{-2}_{p,q}$ such that $dA =0$, then $\mu(A) = kq \beta^2 + kp \gamma$ for some $k \in \mathbb{Z}$.  
\end{lem}
\begin{proof}
	Lemma \ref{lemma 3} gives 
	\[(0,0) = f(dA) = g(\mu(A)).\]
	If $\mu(A) = A_\beta \beta^2 + A_\gamma \gamma$, then 
	\[g(\mu(A)) = A_{\beta} (-p,p) + A_{\gamma} (q,-q) = (0,0).\]
	Since $p$ and $q$ are relatively prime, $A_\beta = kq, A_{\gamma}=kp$ for some $k \in \mathbb{Z}$. 	
\end{proof}

\subsection{Comparison of $\mathcal{C}_{p,q_1}$ and $\mathcal{C}_{p,q_2}$}
\label{subsect comparison}
In Section \ref{subsect comparison}, we prove Lemma \ref{prop odd} which is Theorem \ref{thm quasi to homotopy type} for the case of odd $p$.
We note that for $\mathcal{C}_{p,q_i}$, we use the notation from the previous sections together with a subscript except $\mu_{a,b,c}$.
For example, we use $x_i, y_i, z_i$ instead of $x, y, z$, and for $\mu$, we use $\mu_i$. 
For $\mu_{a,b,c}$, we use $\mu^i_{a,b,c}$	.

\begin{prp}[Theorem \ref{thm quasi to homotopy type} for odd $p$]
	\label{prop odd}
	For any odd $p$, if $\mathcal{C}_{p,q_1}$ and $\mathcal{C}_{p,q_2}$ are quasi-equivalent, then 
	\[\pm q_2\equiv a^2q_1 \hspace{0.5em} (\operatorname{mod} \hspace{0.5em} p)\]
	for some $a\in\Z$.
\end{prp}
\begin{proof}
	First, note that $\cC_{p,q_1}$ and $\cC_{p,q_2}$ are semifree dg algebras, hence they are cofibrant objects (in the model category where the weak equivalences are quasi-equivalences). Then by Proposition \ref{prp:model-hom} and Proposition \ref{prp:weak-equiv}, since $\cC_{p,q_1}$ and $\cC_{p,q_2}$ quasi-equivalent, there exists a quasi-equivalence $F:\mathcal{C}_{p,q_1} \to \mathcal{C}_{p,q_2}$.
	The composition $\mu_2 \circ F$ is a dga morphism from $\mathcal{C}_{p,q_1}$ to $\mathcal{D}$. 
	By Section \ref{subsect Z module morphisms}, there exist $a, b, c \in \mathbb{Z}$ such that 
	\begin{gather}
		\label{eqn commuting}
		\mu_2 \circ F = \mu^1_{a,b,c}.
	\end{gather}
	
	Since $H^*F\colon H^*_{p,q_1}\to H^*_{p,q_2}$ is an isomorphism of graded $\Z$-modules, there is a closed element $A \in \mathcal{C}^{-2}_{p,q_1}$ such that $[F(A)]$ is the same as $[\chi_2]$ in the $H^{-2}_{p,q_2}$. 
	Then,
	\begin{align*}
		q_2 \beta^2 + p \gamma &\stackrel{(i)}{=} [\mu_2(\chi_2)] \\ 
		&\stackrel{(ii)}{=} [\mu^1_{a,b,c}(A)] \\
		&\stackrel{(iii)}{=} [(\tilde{\mu}_{a,b,c} \circ \mu^1_{1,1,0}) (A)] \\
		&\stackrel{(iv)}{=} [\tilde{\mu}_{a,b,c}(kq_1\beta^2 + kp\gamma)] \\
		&\stackrel{(iv)}{=}(kq_1a^2+kpc) \beta^2 +kpb\gamma.
	\end{align*}
	More precisely, the definitions of $\chi_2$ and $\mu^2$ give (i) together with the fact that the differential of $\mathcal{D}$ is zero, Equation \eqref{eqn commuting} gives (ii), (iii) comes from Lemma \ref{lemma 1}, (iv) comes from Lemma \ref{lemma 5}, the definition of $\tilde{\mu}_{a,b,c}$ in the proof of Lemma \ref{lemma 1} gives (v).
	
	Thus, $kpb =p$, or equivalently, $kb=1$, and $q_2=kq_1a^2 + kpc$. 
	The first equation induces that $k = \pm1$, and this concludes that 
	\[\pm q_2\equiv a^2q_1 (\operatorname{mod} \hspace{0.5em} p).\] 
\end{proof}

\subsection{The Case of Even $p$}
\label{subsect even p}
In the current subsection, we discuss the case of even $p$. 

Let assume that $p$ is even.
Then, in Equation \eqref{eqn dga morphism}, $\mu(x)$ could be $\pm 1$. 
This is because $\mu(x)$ is an integer satisfying $\big(\mu(x)\big)^p =1$. 
If $p$ is an even number, $\mu(x)$ is either $1$ or $-1$. 

Since there is a choice for $\mu(x)$, even for fixed $a,b,c \in \mathbb{Z}$, Equation \eqref{eqn dga morphism} is not enough to define a dga morphism.
Thus, the notation $\mu_{a,b,c}$ does not make sense. 
Also, the notation $\mu$ does not make sense too, since $\mu$ is defined as $\mu_{1,1,0}$. 

The notation can be remedied as follows:
Let $\mu_{a,b,c,}$ be the dga morphism such that 
\begin{gather}
	\label{eqn dga morphism for even}
	\mu(x) = 1, \mu(y) = a \beta, \mu(z) = b \gamma + c \beta^2.
\end{gather} 

With the notation $\mu_{a,b,c}$ (resp.\ $\mu$), Lemma \ref{lemma 1} (resp.\ Lemma \ref{lemma 3}) holds for even $p$. 
Moreover, Proposition \ref{lemma 5} also holds for even $p$.  

\begin{prp}[Theorem \ref{thm quasi to homotopy type} for even $p$]
	\label{prop even} 
	For any even $p$,  
	if $\mathcal{C}_{p,q_1}$ and $\mathcal{C}_{p,q_2}$ are quasi-equivalent, then 
	\[\pm q_2\equiv a^2q_1 \hspace{0.5em} (\operatorname{mod} \hspace{0.5em} p)\]
	for some $a\in\Z$.
\end{prp}
\begin{proof}
	Let $F: \mathcal{C}_{p,q_1} \to \mathcal{C}_{p,q_2}$ be a quasi-equivalence. 
	As similar to the proof of Proposition \ref{prop odd}, we consider $\mu_2 \circ F$. 
	This is a dga map, thus, $(\mu_2 \circ F)(x) = \pm 1$. 
	If $(\mu_2 \circ F)(x) = 1$, the proof of Lemma \ref{lemma 5} works for the case of even $p$. 
	Thus, let assume that $(\mu^2 \circ F)(x) = -1$. 
	
	We define a dga map $\delta : \mathcal{C}_{p,q_1} \to \mathcal{C}_{p,q_1}$ such that 
	\[\delta(x_1) = - x_1, \delta(y_1) = y_1, \delta(z_1)= z_1.\]
	It is easy to check that the above equations define a dga isomorphism $\delta$. 
	Thus, $F \circ \delta$ is a quasi-equivalence between $\mathcal{C}_{p,q_1}$ and $\mathcal{C}_{p,q_2}$.
	
	By definition, $(\mu_2 \circ F \circ \delta)(x_1) =1$.
	By considering $F \circ \delta$ instead of $F$, one can prove Proposition \ref{prop even}.
\end{proof}

\begin{proof}[Proof of Theorem \ref{thm quasi to homotopy type}]
	Propositions \ref{prop odd} and \ref{prop even} are Theorem \ref{thm quasi to homotopy type} for the case of odd and even $p$ respectively. 
\end{proof}

\clearpage
\bibliographystyle{amsalpha}
\bibliography{hocolim}

\end{document}